\documentclass[12pt,reqno] {article} 

\usepackage{mathrsfs,amsfonts,amssymb,amsmath}
\usepackage{enumerate}
\usepackage{graphicx,cite}
\usepackage{hyperref}
\hypersetup{colorlinks=true, linkcolor=blue, citecolor=red}

\numberwithin{equation}{section}

\newtheorem{theorem}{Theorem}[section]

\newtheorem{lemma}[theorem]{Lemma}

\newtheorem{definition}[theorem]{Definition}
\newtheorem{remark}[theorem]{Remark}

\newtheorem{convention}[theorem]{Convention}

\newcommand{\qed}{\hfill \mbox{\raggedright \rule{.07in}{.1in}}}

\newenvironment{proof}{\vspace{1ex}\noindent{\bf
Proof}\hspace{0.5em}}{\hfill\qed\vspace{1ex}}

\begin{document}

\title
{Random Young Towers and Quenched Limit Laws}

\author{Yaofeng Su\thanks{Department of Mathematics, University of Houston,
    Houston, TX 77204-3008, USA. \texttt{yfsu@math.uh.edu}}}


\date{\today}


\maketitle

\begin{abstract}
We obtain quenched almost sure invariance principle (with convergence rate) for random Young tower. Applications are some i.i.d. perturbations of some non-uniformly expanding maps. In particular, we answer one open question in \cite{BBM}. 
\end{abstract}

\tableofcontents

%
\section{Introduction}\ \par
 A collection $(\Omega, \mathbb{P}, \sigma, (\Delta_{\omega})_{\omega \in \Omega}, (\mu_{\omega})_{\omega \in \Omega}, (F_{\omega})_{\omega \in \Omega})$ is called a random dynamical system (RDS) if $(\Omega, \mathbb{P})$ is a probability space, $\sigma: \Omega \to \Omega$ is an invertible $\mathbb{P}$-preserving transformation. For each $\omega\in \Omega$, the probability space $(\Delta_{\omega}, \mu_{\omega})$ is called a fiber at $\omega$. $F_{\omega}: \Delta_{\omega} \to \Delta_{\sigma \omega}$ is called a fiber map, satisfying $(F_{\omega})_{*}\mu_{\omega}=\mu_{\sigma \omega}$. $(\mu_{\omega})_{\omega \in \Omega}$ are called equivariant probability measures. \par 
 A decreasing series $\rho_n \searrow 0$ is called an almost surely mixing rate for the RDS if for a.e. $\omega \in \Omega$, there is a constant $C_{\omega}>0$ and a Banach space $B_{\omega} \subset L^1(\Delta_{\omega}, \mu_{\omega})$ s.t. for any $n \in \mathbb{N}$, $ \phi_{\omega} \in B_{\omega}, \Psi_{\sigma^n \omega} \in L^{\infty}(\Delta_{\sigma^n \omega}, \mu_{\sigma^n \omega})$, there is $C_{\phi, \Psi}>0$ and 
 
 \[|\int \phi_{\omega} \cdot \Psi_{\sigma^n \omega} \circ F^n_{\omega} d\mu_{\omega}-\int \phi_{\omega}d\mu_{\omega} \int \Psi_{\sigma^n \omega} d\mu_{\sigma^n \omega}| \le C_{\phi, \Psi} \cdot C_{\omega} \cdot \rho_n \to 0,\]
 
 where $\phi(\omega, \cdot):=\phi_{\omega}(\cdot), \Psi(\omega, \cdot):=\Psi_{\omega}(\cdot)$.
 
 We call the RDS has an uniform almost surely mixing rate if $\operatorname*{ess~sup}_{\omega \in \Omega} C_{\omega} < \infty$, the RDS has a non-uniform almost surely mixing rate if $\operatorname*{ess~sup}_{\omega \in \Omega} C_{\omega} = \infty$.\par
 
A random Young tower (RYT) is a powerful tool to study almost surely mixing rate for RDS with weak hyperbolicity. The original one is constructed by Baladi, Benedicks and Maume-Deschamps \cite{BBM} to obtain the almost surely mixing rate for i.i.d. translations of unimodal maps. In recent years, RYT has been extended and used intensively: Du \cite{D} extends \cite{BBM} to a more general RYT and applies it to i.i.d. perturbations of a wider class of unimodal maps. Bahsoun, Bose, Li, Ruziboev and Vilarinho \cite{LV,BBR} use RYT to obtain almost surely mixing rates for i.i.d. perturbations of some non-uniformly expanding maps. 

In \cite{BBM}, Baladi, Benedicks and Maume-Deschamps ask: if an almost surely mixing rate exists, does quenched central limit theorem (QCLT) hold for the RDS? Furthermore, we can ask: quenched almost surely invariance principle (QASIP), which is a very strong form of quenched limit law strengthening QCLT, quenched functional central limit theorem (QFCLT) and quenched law of iterated logarithm (QLIL), holds for the RDS?

For a RDS with an uniform almost surely mixing rate, there is already extensive literature studying quenched limit laws, see \cite{ALM, OH, HS, DFGV, S1, S2}. However, it is quite natural and more often to see a RDS (e.g. RYT) has a non-uniform almost surely mixing rate. To the best of our knowledge, for such RDS, two papers \cite{AA, K} do make progress: Abdelkader and  Aimino \cite{AA} study RDS with expanding in average. Inspired by \cite{ALM}, they fix one reference measure, instead of finding equivariant probability measures, to study QCLT. But they finally find an example which fails to have QCLT no matter how fast its almost surely mixing rate is. The correct approach is due to Kifer \cite{K}: assuming equivariant probability measures for RDS. His method is introducing hitting time on $(\Omega, \mathbb{P}, \sigma)$ to induce a new RDS which has an uniform almost surely mixing rate. Then he places several conditions on the hitting time so that the QCLT of the induced RDS can be transferred to the original RDS. However, his conditions are quite complicated and hard to verify. In addition, as he remarks in Proposition 2.2 and Remark 6.5 in \cite{K}, his method has to work on specific cases with an explicit representation. Even in his applications under the assumption that $(\Omega, \sigma)$ is Bernoulli, it is still unclear whether the RDS in \cite{K} has QCLT or not. \par

In this paper we will not adopt Kifer's method. We will claim clearly that the RYT has QASIP and answer one open question in \cite{BBM}. Our approach to QASIP is the random version of martingale approximation similar to \cite{Li}. Technical lemmas will be given to overcome the difficulty caused by the unbounded $C_{\omega}$ and obtain the convergence rate of QASIP. This convergence rate is $O(n^{\frac{1}{4}+\epsilon})$ when the RDS has almost surely (stretch) exponential mixing rate. \par

Outline of this paper is as follows: in Section \ref{def} we give the definition of a general RYT in \cite{D} and our main theorem for QASIP. In section \ref{revisit} we revisit RYT in \cite{D} and modify/improve some inequalities such that they can be used to prove QASIP. In section \ref{tech} several technical lemmas are given. In section \ref{proof} we prove our main theorem. In section \ref{fall} we project QASIP from RYT down to the RDS with induced Markov structure via semiconjugacies. In section \ref{app},  Applications, we apply our result to some i.i.d. perturbations of some non-uniformly expanding maps.

\section{Definitions, Main Theorem and Conventions}\label{def}

\begin{definition}[Random Young Tower, see \cite{D}]\label{ryt}\ \par
We construct a random Young tower $(\Delta, F)$  in the following steps:

\begin{enumerate}
    \item Fix a probability space $(\Lambda, \mathcal{B}, m)$, a Bernoulli scheme $(\Omega, \mathbb{P}, \sigma):=(S^{\mathbb{Z}}, {\nu}^{\mathbb{Z}}, \sigma)$ where $\nu$ is probability on the measurable space $S$ and $\sigma$ is invertible left shift on $S^{\mathbb{Z}}$.

    \item Assume for a.e. $\omega \in \Omega$, there is a countable partition $ \mathcal{P}_{\omega} $ of a full measure subset $\mathcal{D}_{\omega}$ of $\Lambda$ and a function $R_{\omega}: \Lambda \to \mathbb{N}$ such that $R_{\omega}$ is constant on each $U_{\omega} \in \mathcal{P}_{\omega}$.
    \item  Assume $R_{\omega}(x)$ is a stopping time: if $R_{\omega}(x) = n$ and $\omega_i = \omega_i'$ for $0 \le i<n$, then $R_{\omega'}(x) = n$.
    \item For a.e. $\omega \in \Omega$, $l \in \mathbb{N}$, define $\Delta_{\omega, 0}:= \Lambda$ and the $l$-th level by
    \[\Delta_{\omega, l}:=  \{(x,l) : x \in \Lambda, R_{\sigma^{-l}\omega}(x)> l\}.\]
    
    Also define a tower at $\omega$ by
    
    \[\Delta_{\omega}:= \bigcup_{l \ge 0} \Delta_{\omega, l}.\]
    
     $\Delta_{\omega}$ is endowed with a measure $m_{\omega}$, a $\sigma$-algebra $\mathcal{B}_{\omega}$ and a partition  $\mathcal{Z}_{\omega}$ naturally from the probability space $(\Lambda, \mathcal{B}, m)$ and the partitions $(\mathcal{P}_{\sigma^{-l}\omega})_{l\ge 0}$. 
     
     \item Assume for a.e. $\omega \in \Omega$, there is $F_{\omega}: \Delta_{\omega} \to \Delta_{\sigma \omega} $ satisfying: if  $R_{\sigma^{-l}\omega}(x)> l+1$,  then $ F_{\omega}(x,l)=(x,l+1)$. If  $R_{\sigma^{-l}\omega}(x)= l+1$ and $x \in U_{\sigma^{-l}\omega} \in \mathcal{P}_{\sigma^{-l}\omega}$, then $F_{\omega}$ maps $U_{\sigma^{-l}\omega} \times \{l\} $ bijectively onto $\Delta_{\sigma \omega, 0}$. 
    \item Define $F_{\omega}^n:=F_{\sigma^{n-1}\omega}\circ F_{\sigma^{n-2}\omega} \circ \cdots \circ F_{\sigma \omega} \circ F_{\omega}$, assume the partition $\mathcal{Z}_{\omega}$ is generating for $F_{\omega}$ in the sense that the diameter of the partition $\bigvee^n_{j=0} (F^j_{\omega})^{-1}\mathcal{Z}_{\sigma^j \omega}$ tends to $0$ as $n \to \infty$.
    \item Assume for a.e. $\omega \in \Omega, m_{\omega}(\Delta_{\omega})< \infty$.
    \item Assume there is an integer $M \in \mathbb{N}$, $\{\epsilon_i > 0,i = 1,\cdots,M\}$ and $\{t_i\in \mathbb{N},i= 1,\cdots,M\}$ with $\gcd(t_i) = 1 $ such that for a.e. $\omega \in \Omega$, all $1 \le i \le M$,
    \[m(x \in \Lambda: R_{\omega}(x) = t_i) > \epsilon_i.\]
    \item  Extend $R_{\omega}$ to $\Delta_{\omega}$ (still denoted by $R_{\omega}$): for any $(x,l) \in \Delta_{\omega}$,
     \[R_{\omega}(x,l):=R_{\sigma^{-l}\omega}(x)-l.\]
     Define n-th return time on $\Delta_{\omega}$ inductively: for any $x\in \Delta_{\omega}$, 
    \[R^0_{\omega}(x):=0, R^1_{\omega}(x):=R_{\omega}(x), R^n_{\omega}(x):=R^{n-1}_{\omega}(x)+R_{\sigma^{R^{n-1}_{\omega}(x)} \omega}(F_{\omega}^{R^{n-1}_{\omega}}(x)).\]
    Define separation time $s_{\omega}: \Delta_{\omega} \times \Delta_{\omega} \to \mathbb{N} \cup \{\infty\}$ by:
    \[s_{\omega}(x,y)=\inf \{n: F_{\omega}^{R^n_{\omega}(x)}(x), F_{\omega}^{R^n_{\omega}(y)}(y) \text{ lie in different elements of } \mathcal{Z}_{\sigma^{R^n_{\omega}(x)} \omega}\}.\]
    
    Assume there is a constant $C_{F} > 0$ and $ \beta \in (0,1)$ such that for a.e. $\omega \in \Omega$ and each element $J_{\omega}\in \mathcal{Z}_{\omega}$, the map $F^{R_{\omega}}_\omega|_{J_{\omega}}$ and its inverse are non-singular w.r.t. $m$, and for each $x,y \in J_{\omega}$, 
    \begin{equation}\label{towerdistortion}
      |\frac{JF^{R_{\omega}}_\omega(x)}{JF^{R_{\omega}}_\omega(y)}-1|\le C_F \cdot \beta^{s_{\sigma^{R_{\omega}(x)} \omega}(F^{R_{\omega}}_\omega(x), F^{R_{\omega}}_\omega(y))} .  
    \end{equation}

    \item Assume there is a constant $C>0$ s.t. 
    \[\int  m(x \in \Lambda: R_{\omega}(x)>n) d\mathbb{P} \le C \cdot \rho_n, \]
    
    where $\rho_n:= e^{-a\cdot n^{b}}$ or $\frac{1}{n^{D}}$ for some constants $a>0, b \in (0,1], D>4$.
    
    \item Define the random Young tower
    $(\Delta, F)$ by
    \[\Delta:= \bigcup_{\omega \in \Omega} \{\omega\} \times \Delta_{\omega}, F(\omega, x):=(\sigma \omega, F_{\omega}x).\]

\end{enumerate}
\end{definition}

\begin{remark}  

For the 3D picture of the RYT's dynamic, see Figure 1 in \cite{BBR}.
    
\end{remark}

\begin{definition}[Dynamical  Lipschitz Cone]\label{dlc}\ \par

\[\mathcal{F}_{\beta}^{+}:=\{\phi: \Delta \to \mathbb{C}| \text{ there is } C_{\phi}>0, \text{ for any } J_{\omega} \in \mathcal{Z}_{\omega}, \text{ either } \phi_{\omega}|_{J_{\omega}}=0 \text{ or }\]

\[\phi_{\omega}|_{J_{\omega}}>0 \text{ and for any } x,y \in J_{\omega}, |\log \frac{\phi_{\omega}(x)}{\phi_{\omega}(y)}| \le C_{\phi} \cdot \beta^{s_{\omega}(x,y)}\},\]
where $C_{\phi}$ is called Lipschitz constant for $\phi$.
\end{definition}

\begin{definition}[Bounded Random Lipschitz Function]\label{brlf}\ \par
For any $p \in (1,\infty]$, define:
\[\mathcal{F}_{\beta,p}^{\mathcal{K}}:=\{\phi: \Delta \to \mathbb{C}| \text{ there are constants }  C_{\phi}>0, \mathcal{K}_{\omega} \ge 1 \text{ such that } \]

\[\mathcal{K}_{(\cdot)} \in L^p(\Omega), |\phi_{\omega}(x)|\le C_{\phi} \text{ and } |\phi_{\omega}(x)- \phi_{\omega}(y)| \le C_{\phi} \cdot \mathcal{K}_{\omega} \cdot \beta^{s_{\omega}(x,y)}\},\]
where $C_{\phi}$ is also called Lipschitz constant for $\phi$.
\end{definition}

\begin{remark} Instead of defining random bounded function as the Section 2.1.2 of \cite{D} does, we only define bounded random Lipschitz function in Definition \ref{brlf} for our main purpose in this paper: quenched limit law.

\end{remark}

\begin{theorem}[QASIP for RYT]\label{QASIP}\ \par
Assume the RYT in Definition \ref{ryt}, then for a.e. $\omega \in \Omega$, there are equivariant probability measures $(\mu_{\omega})_{\omega \in \Omega}$ on $(\Delta_{\omega})_{\omega \in \Omega}$, that is,

\[(F_{\omega})_{*} \mu_{\omega}=\mu_{\sigma \omega}.\]

Moreover, for any $\phi \in \mathcal{F}_{\beta,p}^{\mathcal{K}}$ with fiberwise mean $0$: $\int \phi_{\omega} d\mu_{\omega}=0$. Define:

\[ \sigma_n^2({\omega}): = \int(\sum_{k\le n} \phi_{\sigma^k \omega} \circ F^k_{\omega} )^2 d\mu_{\omega}.\] 

Assume $\rho_n$ in Definition \ref{ryt} is $ e^{-a\cdot n^{b}}$ or $\frac{1}{n^{D}}$ for some constants $a>0, b \in (0,1], D>2+\frac{4p}{p-1}$. Then RYT $(\Delta, F)$ satisfies the following:
\begin{enumerate}
    \item There is a constant $\sigma^2 \ge 0$ s.t. $\lim_{n \to \infty}\frac{\sigma^2_n(\omega)}{n}=\sigma^2$ a.e. $\omega \in \Omega$.
    \item If $\sigma^2>0$, we have QASIP: there is $\epsilon_0 \in (0,\frac{1}{4})$ s.t. for a.e. $\omega \in \Omega$, there is a Brownian motion $B^{\omega}$ defined on some extension of probability space $(\Delta_{\omega}, \mu_{\omega})$, say $\bf{\Delta_{\omega}}$, such that:

\begin{equation}\label{matching}
\sum_{k\le n} \phi_{\sigma^k \omega} \circ F^k_{\omega} -B^{\omega}_{\sigma^2_{n}(\omega)}=O(n^{\frac{1}{4}+\epsilon_0}) \text{ a.s.,}
\end{equation}

the constant indicated in $O(\cdot)$ depends on $\omega$ and $ x \in \bf{\Delta_{\omega}} $, and the term a.s. means almost surely w.r.t. the probability on $\bf{\Delta}_{\omega}$. Moreover, we have an explicit formula for the convergence rate $O(n^{\frac{1}{4}+\epsilon_0})$:

\begin{enumerate}
\item if $\rho_n=e^{-a\cdot n^b}$, $\epsilon_0>0$ can be chosen to be any small number,
\item if $\rho_n=\frac{1}{n^D}$, $\epsilon_0$ can be chosen to be any number between $(\epsilon_D, \frac{1}{4})$, where
\[\epsilon_D=\max \{\frac{1}{4}+\frac{3\epsilon_1-2\epsilon_1^3-\epsilon_1^2}{4}, \epsilon_1, \frac{1+\epsilon_1}{4}\}-\frac{1}{4}\]
and 
\[\epsilon_1=\frac{2p}{(p-1)(D-2)}\in (0,\frac{1}{2}).\]
\end{enumerate}

\item If $\sigma^2=0$, we have Coboundary: define $\mu$ on $\Delta$: 
\[\mu(A):=\int \mu_{\omega}(A_{\omega}) d\mathbb{P},\]

where $A$ is measurable on $\Delta$ and $A_{\omega}:=\{x \in \Delta_{\omega}: (x,\omega) \in A\}$. Then there is measurable function $g$ on $\Delta$ s.t.
\[\phi_{\sigma \omega} \circ F_{\omega}(x)=g_{\sigma \omega} \circ F_{\omega}(x)-g_{\omega}(x) \text{ a.s.-}\mu.\]
Moreover, if $\rho_n=e^{-a\cdot n^b}$, $g \in L^{\infty}(\Delta, \mu)$; if $\rho_n=\frac{1}{n^D}$, $g \in L^{\frac{(D-2-\delta)\cdot(p-1)}{(1+\delta)p}}(\Delta, \mu)$ for sufficiently small $\delta>0$.

\end{enumerate}

\end{theorem}

\begin{remark}\label{qasipremark}\ \par

\begin{enumerate}

\item In our section \ref{app}, Applications, $p=\infty$.

\item For any $ n \ge 1$, define $S_n^{\omega}:= \sum_{k\le n} \phi_{\sigma^k \omega} \circ F^k_{\omega}$, and piecewise continuous function $S^{n, \omega}$ on $[0,1]$:

\[S^{n, \omega}_t: =\frac{S^{\omega}_{i-1}}{\sqrt n}+\frac{t-\frac{i-1}{n}}{\frac{i}{n}-\frac{i-1}{n}} \cdot \frac{S^{\omega}_{i}-S^{\omega}_{i-1}}{\sqrt n}, t \in [\frac{i-1}{n},\frac{i}{n}], \text{ where } 1 \le i \le n, \frac{0}{0}:=0. \]

Then QASIP for RYT implies QCLT, QLIL, QFCLT for RYT respectively, that is, there is a constant $\sigma^2>0$, for a.e. $\omega \in \Omega$,

\[\frac{S_n^{\omega}}{\sqrt n} \xrightarrow[\mu_{\omega}]{d} N(0,\sigma^2),\]
\[\limsup_{n\to \infty}\frac{S_n^{\omega}}{ \sqrt{n\log\log n}}=\sigma \text{ a.s.-}\mu_{\omega},\]
\[\liminf_{n\to \infty}\frac{S_n^{\omega}}{\sqrt{n\log\log n}}=-\sigma \text{ a.s.-}\mu_{\omega},\]

\[S^{n, \omega} \xrightarrow[\mu_{\omega}]{d} \sigma  \cdot B \text{ on } C[0,1],\]

where  $B$  is standard one-dimensional Brownian motion. 

\item For RYT with polynomially $\rho_n=\frac{1}{n^D}$, \cite{D} obtains mixing rate for RYT when $D>4$, while \cite{BBR} obtains mixng rate with a wider range $D>1$, under two more restrictive conditions: $(P6)$ and $(P7)$ in \cite{BBR}. Depending on the RDS we study, restrictive conditions could be added to obtain much stronger results. But in this paper, we just consider the general RYT studied in \cite{D}, since the restrictive RYT in \cite{BBR} in the range $D>4$ is just a special case of the general RYT. We believe QASIP for the restrictive RYT in \cite{BBR} holds with a little wider range of $D$.

\end{enumerate}

\end{remark}

\begin{convention}\ \par
\begin{enumerate}
    \item $C_{a}$ means a constant $C$ depending on $a$,
    \item $\mathbb{E}_{\mu_{\omega}} $ means the expectation w.r.t. $\mu_{\omega}$, $\mathbb{E}$ means expectation of $\mathbb{P}$.
    \item We do not specify the $\sigma$-algebra of a measure space if it could be naturally understood.
    \item $a_n=O_a(b_n)$ means: there is $C_a>0$ s.t. $a_n \le C_a \cdot b_n$ for all $n \in \mathbb{N}$. 
    \item $a_n=C_a^{\pm} \cdot b_n$ means: there is $C_a\ge 1$ s.t. $C_a^{-1} \cdot b_n \le a_n \le C_a \cdot b_n$.
\end{enumerate}
\end{convention}

\section{Revisit Random Young Tower}\label{revisit}
\begin{lemma}[AC Equivariant Probability and Matching]\label{acma}\ \par

For RYT in Definition \ref{ryt}, we have: for a.e. $\omega \in \Omega$, there is an unique absolutely continuous equivariant probability:
\[d\mu_{\omega}:=h_{\omega}dm_{\omega} \text{ on } \Delta_{\omega},\]  
\[h(\omega, \cdot):=h_{\omega}(\cdot) \text{ on } \Delta\]

satisfying:

\begin{equation}\label{equidens}
  (F_{\omega})_{*}\mu_{\omega}=\mu_{\sigma \omega},  h \in \mathcal{F}^{+}_{\beta}, \operatorname*{ess~sup}_{\omega \in \Omega} h_{\omega} < \infty, h_{\omega}>0, 
\end{equation}

\[(\Delta, F, \mu) \text{ is exact, mixing and ergodic}.\]

In addition, there is an integer $l_0 > 0$ such that for any $l \ge l_0$, there is $ \epsilon_l\in (0, 1)$, for a.e. $\omega\in \Omega$, $m_{\omega}(\Delta_{\omega, 0}\cap F_{\omega}^{-l}\Delta_{\sigma^l \omega, 0})>\epsilon_{l}$.

Furthermore, we have the following matching: let $\lambda_{\omega}, \lambda'_{\omega}$  be absolutely continuous probability measures on $\Delta_{\omega}$, with desities: $\frac{d\lambda_{\omega}}{dm_{\omega}}, \frac{d\lambda'_{\omega}}{dm_{\omega}} \in \mathcal{F}^{+}_{\beta}$. Recall $R^i_{\omega}(x)$ is the i-th return time of $(\omega, x)\in \Delta$ into the 0-th level. Define return times on  $\bigcup_{\omega \in \Omega}\{\omega\}\times \Delta_{\omega} \times \Delta_{\omega}$ alternatively and recursively: 

\[\tau_0^\omega(x, x'):=0, \tau_1^\omega(x, x'):=R^{l_0}_{\omega}(x),\]
\[\tau_2^\omega(x, x'):=\tau_1^\omega(x, x')+R^{l_0}_{\sigma^{\tau^{\omega}_1(x,x')}\omega}( F_{\omega}^{\tau^{\omega}_1(x,x')} x'),\]
\[\tau_3^\omega( x, x'):=\tau_2^\omega(x, x')+R^{l_0}_{\sigma^{\tau^{\omega}_2(x,x')}\omega}( F_{\omega}^{\tau^{\omega}_2(x,x')} x),\]
\[\tau_4^\omega(x, x'):=\tau_3^\omega(x, x')+R^{l_0}_{\sigma^{\tau^{\omega}_3(x,x')}\omega}( F_{\omega}^{\tau^{\omega}_3(x,x')} x'), \]
\[\tau_5^\omega( x, x'):=\tau_4^\omega(x, x')+R^{l_0}_{\sigma^{\tau^{\omega}_4(x,x')}\omega}( F_{\omega}^{\tau^{\omega}_4(x,x')} x),\]
\[\cdots \cdots\]
\[T^{\omega}(x, x'):=\min\{\tau^{\omega}_i(x,x'), i \ge 1: (F_{\omega}\times F_{\omega})^{\tau^{\omega}_i(x,x')}(x,x') \in \Delta_{\sigma^{\tau^{\omega}_i(x,x')}\omega,0}\times \Delta_{\sigma^{\tau^{\omega}_i(x,x')}\omega,0}\},\]
\[T_{0}^{\omega}:=0, T_{1}^{\omega}:=T^{\omega},\]
\[T_{n}^{\omega}(x,x'):=T_{n-1}^{\omega}(x,x')+ T^{\sigma^{T^{\omega}_{n-1}(x,x')}\omega}((F_{\omega}\times F_{\omega})^{T^{\omega}_{n-1}(x,x')}(x,x')).\]

Then there is a constant $C>0, r \in (0,1)$, for a.e. $\omega \in \Omega$, 

\begin{equation}\label{couple}
 |(F^n_{\omega})_{*}\lambda_{\omega}-(F^n_{\omega})_{*}\lambda'_{\omega}|\le C \cdot \sum_{i=0}^{\infty}r^i \cdot (\lambda_{\omega} \times \lambda'_{\omega})(T^{\omega}_i \le n <  T^{\omega}_{i+1}),   
\end{equation}
where $C, r$ only depend on $\beta, C_{F}$ and Lipschitz constants of  $\frac{d\lambda_{\omega}}{dm_{\omega}}, \frac{d\lambda'_{\omega}}{dm_{\omega}}$.

If $\rho_n=e^{-a\cdot n^{b}}$ or $\frac{1}{n^{D}}$ where $a>0, b \in (0,1], D>4$, then
\[\int  m^{\otimes 2}_{\omega}  (\Delta_{\omega}\times \Delta_{\omega})d\mathbb{P}< \infty.\]

In addition, for any small $ \delta>0$, there are $C=C_{\beta,F,\delta}>0$ and sufficiently small $\alpha=\alpha_{\delta}>0$ s.t.
\begin{equation}\label{coupletail}
\int m^{\otimes 2}_{\omega}(T^{\omega}_{\lfloor n^{\alpha}\rfloor}>n) d\mathbb{P}\le C \cdot \frac{1}{n^{D-2-\delta}}.
\end{equation}

\end{lemma}

\begin{proof}
For the proof of (\ref{equidens}), see Theorem 2.2.1, Proposition 2.3.1, 2.3.3-2.3.4 \cite{D}. For the proof of (\ref{couple}), see Theorem 3.1.1 \cite{D}.

Since $\rho_n\le e^{-a\cdot n^b}$ implies $\rho_n\le \frac{1}{n^D}$ for arbitrary large $D$, so we just take care of the latter case, and refer Corollary 7.1.2 \cite{D} for the proof of (\ref{coupletail}). This does not hurt our estimates and calculations of QASIP.
\end{proof}

\begin{lemma}\label{avermatch}
Let $\phi\in \mathcal{F}^{\mathcal{K}}_{\beta,p}$, $\delta>0$ be any small number. Define probability: \[d\lambda_{\omega}:=\frac{\phi_{\omega}+\mathcal{K}_{\omega} \cdot C_{\phi}+2 \cdot C_{\phi}}{\int (\phi_{\omega}+\mathcal{K}_{\omega} \cdot C_{\phi}+2 \cdot C_{\phi})d\mu_{\omega}} d\mu_{\omega},\]

where $C_{\phi}$ is Lipschitz constant for $\phi$. Then there is a constant $C= C_{h,F,\beta,\delta}$,
\begin{equation}\label{coupledecay}
 \int |(F^n_{\omega})_{*}\lambda_{\omega}-(F^n_{\omega})_{*}\mu_{\omega}| d\mathbb{P} \le C \cdot \frac{1}{n^{D-2-\delta}},  
\end{equation}

\end{lemma}
\begin{proof} 
First note that: by (\ref{equidens}), for any $x, y \in \Delta_{\omega}$, recall $\mathcal{K}_{\omega} \ge 1$,  
\begin{equation}\label{upperboundoflambda}
  \frac{d\lambda_{\omega}}{dm_{\omega}}(x) \le  \frac{C_{\phi}\cdot (3+\mathcal{K}_{\omega})}{C_{\phi}\cdot (1+\mathcal{K}_{\omega})} \cdot  \operatorname*{ess~sup}_{\omega \in \Omega} h_{\omega} \le \frac{3+\mathcal{K}_{\omega}}{1+\mathcal{K}_{\omega}} \cdot C_h \le 2 C_{h},  
\end{equation}

where $C_h$ is Lipschitz constant of $h$.

\[|\log \frac{\frac{d\lambda_{\omega}}{dm_{\omega}}(x)}{\frac{d\lambda_{\omega}}{dm_{\omega}}(y)}|\le |\log \frac{h_{\omega}(x)}{h_{\omega}(y)}| +|\log \frac{\phi_{\omega}(x)+\mathcal{K}_{\omega} \cdot C_{\phi}+2 \cdot C_{\phi}}{\phi_{\omega}(y)+\mathcal{K}_{\omega} \cdot C_{\phi}+2 \cdot C_{\phi}}|. \]

Using inequality $\log x\le x-1$ when $x \ge 1$, the above inequality becomes:
\[\le C_{h} \cdot \beta^{s_{\omega}(x,y)}+ \frac{|\phi_{\omega}(x)-\phi_{\omega}(y)|}{\phi_{\omega}(y)+\mathcal{K}_{\omega} \cdot C_{\phi}+2 \cdot C_{\phi}} \]

\[\le  C_{h} \cdot \beta^{s_{\omega}(x,y)}+ \frac{C_{\phi}\cdot \mathcal{K}_{\omega} \cdot \beta^{s_{\omega}(x,y)}}{\phi_{\omega}(y)+\mathcal{K}_{\omega} \cdot C_{\phi}+2 \cdot C_{\phi}}\le (C_{h}+1) \cdot \beta^{s_{\omega}(x,y)}.\]

Therefore, $\frac{d\lambda_{\omega}}{dm_{\omega}} \in \mathcal{F}^{+}_{\beta}$ with Lipschitz constant $2C_{h}+1$.

By (\ref{couple}),  there is a constant $C=C_{\beta, F, h}>0$ and $\alpha=\alpha_{\delta}$ such that 
\[\int |(F^n_{\omega})_{*}\lambda_{\omega}-(F^n_{\omega})_{*}\mu_{\omega}| d\mathbb{P} \le C \cdot \int \sum_{i=0}^{\infty}r^i \cdot (\lambda_{\omega} \times \mu_{\omega})(T^{\omega}_i \le n <  T^{\omega}_{i+1})d\mathbb{P}\]

\[=C \cdot \int \sum_{i=\lfloor n^{\alpha}\rfloor}^{\infty}r^i \cdot (\lambda_{\omega} \times \mu_{\omega})(T^{\omega}_i \le n <  T^{\omega}_{i+1})d\mathbb{P}\]
\[+C \cdot \int \sum_{i=0}^{\lfloor n^{\alpha}\rfloor-1}r^i \cdot (\lambda_{\omega} \times \mu_{\omega})(T^{\omega}_i \le n <  T^{\omega}_{i+1})d\mathbb{P}.\]

Note that $r<1$, by (\ref{upperboundoflambda}), the above inequality becomes:
\[\le C\cdot r^{\lfloor n^{\alpha}\rfloor}+2 \cdot C^2_h \cdot C  \cdot \int (m_{\omega} \times m_{\omega})(T^{\omega}_{\lfloor n^{\alpha}\rfloor}>n) d\mathbb{P}. \]

By (\ref{avermatch}), the above inequality becomes:

\[\le C\cdot r^{\lfloor n^{\alpha}\rfloor}+2\cdot C_h^2 \cdot C  \cdot \frac{1}{n^{D-2-\delta}} \le C \cdot \frac{1}{n^{D-2-\delta}},\]

where the last constant $C$ depends on $\alpha_{\delta}, \delta, \beta, F, h$.

\end{proof}

\begin{definition}[Dual Operator]\label{dual}\ \par
 $P_{\omega}: L^1(\Delta_{\omega}, \mu_{\omega}) \to L^1(\Delta_{\sigma \omega}, \mu_{\sigma \omega})$ is called dual operator for $F_{\omega}: \Delta_{\omega} \to \Delta_{\sigma \omega} $ if it satisfies: for any $\Psi_{\omega} \in L^1(\Delta_{\omega}, \mu_{\omega}), \Upsilon_{\sigma \omega} \in L^{\infty}(\Delta_{\sigma \omega}, \mu_{\sigma \omega})$,
\[\int \Psi_{\omega} \cdot \Upsilon_{\sigma \omega} \circ F_{\omega}d\mu_{\omega} = \int P_{\omega} (\Psi_{\omega}) \cdot \Upsilon_{\sigma \omega} d\mu_{\sigma \omega}. \]
\end{definition}
\begin{lemma}[Property of Dual Operator]\label{produal}\ \par

For RYT, the dual operator $P_{\omega}$ for $F_{\omega}$ exists for a.e. $\omega \in \Omega$. Moreover, for a.e. $\omega \in \Omega$, any $i,k \ge 0$, any measurable functions $\Psi, \Upsilon$ on $\Delta$:

if $\Psi \in L^{\infty}(\Delta, \mu)$, 
\begin{equation}\label{dualbound}
 ||P_{\omega}\Psi_{\omega}||_{L^{\infty}(\mu_{\sigma \omega})}\le ||\Psi_{\omega}||_{L^{\infty}(\mu_{\omega})},
\end{equation}

if $\Psi \in L^{1}(\Delta, \mu)$,
\begin{equation}\label{conditional}
 \mathbb{E}_{\mu_{\omega}}[\Psi_{\sigma^{i} \omega} \circ F^{i}_{\omega}|{(F_{\omega}^{i+1})^{-1}\mathcal{B}_{\sigma^{i+1} \omega}}]=[P_{\sigma^i \omega}(\Psi_{\sigma^i \omega})]\circ F_{\omega}^{i+1} \text{ in } L^{1}(\mu_{\omega}),
\end{equation}
\begin{equation}\label{pushforward}
   \frac{(F^{i}_{\omega})_{*}(\Psi_{\omega} d\mu_{\omega})}{d\mu_{\sigma^i \omega}}=P^i_{\omega} (\Psi_{\omega}) \text{ in } L^1(\mu_{\sigma \omega}),
\end{equation}

if $\Psi, \Upsilon \in L^{2}(\Delta, \mu)$,
\begin{equation}\label{conditional2}
    P_{\omega}^{i+k}(\Psi_{\sigma^i \omega} \circ F_{\omega}^i \cdot \Upsilon_{\omega})=P_{\sigma^i \omega}^k(\Psi_{\sigma^i\omega}\cdot P_{\omega}^i(\Upsilon_{\omega})) \text{ in } L^1(\mu_{\sigma^{i+k}\omega}),
\end{equation}
where $P_{\omega}^i:=P_{\sigma^{i-1}\omega}\circ \cdots \circ P_{\sigma \omega}\circ P_{\omega}$.
\end{lemma}
\begin{proof}
By (\ref{equidens}), $h_{\omega}>0$ a.e. $\omega \in \Omega$. Similar to Ruelle-Perron-Frobenius operator, it is straightforward to verify (\ref{dualbound}-\ref{conditional2}) via Definition \ref{dual} and the existence of $P_{\omega}$: for a.e. $\omega \in \Omega$:
\begin{equation}\label{dualexpress}
(P_{\omega}\Psi_{\omega})(x)=\frac{1}{h_{\sigma \omega}(x)}\sum_{F_{\omega}(y)=x}\frac{\Psi_{\omega}(y)\cdot h_{\omega}(y)}{JF_{\omega}(y)} \text{ in } L^1(\mu_{\sigma \omega}),
\end{equation}
where $J$ is Jacobian of $F_{\omega}$ w.r.t. $m$.   

\end{proof}
\begin{lemma}[Average Decay]\label{averagedecay}\ \par
For RYT, any $\phi \in \mathcal{F}^{\mathcal{K}}_{\beta,p}$ and any small $\delta>0$, there is $C=C_{\phi} \cdot C_{h,F,\beta, \delta,p} \cdot ||\mathcal{K}||_{L^p}$ such that
\[\mathbb{E} \int |P_{\omega}^n(\phi_{\omega}-\int \phi_{\omega}d\mu_{\omega})|d\mu_{\sigma^n\omega} \le C \cdot \frac{1}{n^{(D-2-\delta)\cdot \frac{p-1}{p}}}.\]
\end{lemma}

\begin{proof}

Let 
\[d\lambda_{\omega}:=\frac{\phi_{\omega}+\mathcal{K}_{\omega} \cdot C_{\phi}+2 \cdot C_{\phi}}{\int (\phi_{\omega}+\mathcal{K}_{\omega} \cdot C_{\phi}+2 \cdot C_{\phi})d\mu_{\omega}} d\mu_{\omega},\]

by (\ref{pushforward}):
\[\mathbb{E} \int |P_{\omega}^n(\phi_{\omega}-\int \phi_{\omega}d\mu_{\omega})|d\mu_{\sigma^n \omega}=\int |\int (\phi_{\omega}+2C_{\phi}+ C_{\phi} \mathcal{K}_{\omega} )d\mu_{\omega}| \cdot |(F^n_{\omega})_{*}\lambda_{\omega}-(F^n_{\omega})_{*}\mu_{\omega}| d\mathbb{P} \]

\[\le 3C_{\phi} \cdot \int |(F^n_{\omega})_{*}\lambda_{\omega}-(F^n_{\omega})_{*}\mu_{\omega}| d\mathbb{P}+C_{\phi} \cdot \int \mathcal{K}_{\omega} \cdot |(F^n_{\omega})_{*}\lambda_{\omega}-(F^n_{\omega})_{*}\mu_{\omega}| d\mathbb{P}.\]

Using H\"older inequality, let $\frac{1}{p'}=1-\frac{1}{p}$, the above inequality becomes:

\[\le 3C_{\phi} \cdot \int |(F^n_{\omega})_{*}\lambda_{\omega}-(F^n_{\omega})_{*}\mu_{\omega}| d\mathbb{P}+C_{\phi} \cdot ||\mathcal{K}||_{L^p} \cdot (\int |(F^n_{\omega})_{*}\lambda_{\omega}-(F^n_{\omega})_{*}\mu_{\omega}|^{p'} d\mathbb{P})^{\frac{1}{p'}}.\]

Using the fact that $|(F^n_{\omega})_{*}\lambda_{\omega}-(F^n_{\omega})_{*}\mu_{\omega}|\le 2$ and (\ref{coupledecay}), we have
\[\le 3C_{\phi} \cdot C_{h,F,\beta,\delta} \cdot \frac{1}{n^{D-2-\delta}}+2^{\frac{p'-1}{p'}}C_{\phi} \cdot ||\mathcal{K}||_{L^p} \cdot (\int |(F^n_{\omega})_{*}\lambda_{\omega}-(F^n_{\omega})_{*}\mu_{\omega}| d\mathbb{P})^{\frac{1}{p'}} \]
\[\le 3C_{\phi} \cdot C_{h,F,\beta,\delta} \cdot \frac{1}{n^{D-2-\delta}}+2^{\frac{p'-1}{p'}}C_{\phi} \cdot ||\mathcal{K}||_{L^p} \cdot C_{h,F,\beta,\delta}^{\frac{1}{p'}} \cdot \frac{1}{n^{\frac{D-2-\delta}{p'}}} \]
\[\le C_{\phi} \cdot C_{h,F,\beta, \delta,p} \cdot ||\mathcal{K}||_{L^p} \cdot \frac{1}{(n^{D-2-\delta)\cdot \frac{p-1}{p}}}.\]

\end{proof}

\section{Several Lemmas}\label{tech}
\begin{lemma}\label{momentcontrol}
If $\Psi \in L^q(\Delta, \mu)$ with $ q > 2$, then for any sufficiently small $\delta \in (0, q-2)$, a.e. $\omega \in \Omega$, we have:

\[\int |\Psi_{\sigma^n \omega} \circ F^n_{\omega}|^q d\mu_{\omega}=O_{\omega,q}( n),\]

\[\int |\Psi_{\sigma^n \omega} \circ F^n_{\omega}|^2 d\mu_{\omega}=O_{\omega,q}(n^{\frac{2}{q}}),\]

\[\Psi_{\sigma^n \omega} \circ F^n_{\omega}(x) =O_{\omega,x, \delta}(n^{\frac{2+\delta}{q}}) \text{ a.s. } x \in \Delta_{\omega}.\]

\end{lemma}

\begin{proof}
By ergodic theorem:

\[\frac{\sum_{i\le n}\int |\Psi_{\sigma^i \omega} \circ F^i_{\omega}|^q d\mu_{\omega} }{n} \to \mathbb{E}\int |\Psi_{\omega}  |^q d\mu_{\omega} \text{ a.e. } \omega \in \Omega. \]

So \[\int |\Psi_{\sigma^n \omega} \circ F^n_{\omega}|^q d\mu_{\omega}=O_{\omega,q}( n)\]

and 

\[\int |\Psi_{\sigma^n \omega} \circ F^n_{\omega}|^2 d\mu_{\omega}=O_{\omega,q}(n^{\frac{2}{q}}).\]

Since

\[\int |\frac{\Psi_{\sigma^n \omega} \circ F^n_{\omega}}{n^{\frac{2+\delta}{q}}}|^q d\mu_{\omega}=O_{\omega,q}(\frac{1}{n^{1+\delta}}),\]

by Borel-Cantelli Lemma:
\[\Psi_{\sigma^n \omega} \circ F^n_{\omega}(x) =O_{\omega,x, \delta}(n^{\frac{2+\delta}{q}}) \text{ a.s. } x \in \Delta_{\omega}.\]

\end{proof}

\begin{lemma}[Martingale Convergence Rate]\label{ratemartingale}\ \par

If $\Psi \in L^q(\Delta, \mu)$ with $q>2$, $(\Psi_{\sigma^n\omega}\circ F^n_{\omega})_{n \ge 0}$ is (reverse) martingale difference, then for any sufficiently small $\delta>0 $, a.e. $\omega \in \Omega$:

\[||\sum_{i \le n}\Psi_{\sigma^i \omega} \circ F^i_{\omega}||_{L^q(\mu_{\omega})} = O_{\omega}(n^{\frac{1}{2}}),\]

\[\sum_{i \le n}\Psi_{\sigma^i \omega} \circ F^i_{\omega}(x)=O_{x,\omega,q, \delta}(n^{\frac{1}{2}+\frac{1+\delta}{q}}) \text{ a.s.-}\mu_{\omega}.\]
\end{lemma}
\begin{proof}
By Burkholder-Davis-Gundy inequality and Minkowski inequality, there is a constant $C_q$ s.t.

\[||\sum_{i \le n}\Psi_{\sigma^i \omega} \circ F^i_{\omega}||_{L^q(\mu_{\omega})} \le C_q \cdot ||(\sum_{i \le n}\Psi^2_{\sigma^i \omega} \circ F^i_{\omega})^{\frac{1}{2}}||_{L^q{(\mu_{\omega})}} \le C_q \cdot \sqrt{\sum_{i \le n}||\Psi^2_{\sigma^i \omega}\circ F_{\omega}^i||_{L^{\frac{q}{2}}(\mu_{\omega})}}.\]
Since $\mathbb{E}||\Psi^2_{\omega}||_{L^{\frac{q}{2}}(\mu_{\omega})}\le (\mathbb{E}\int |\Psi_{\omega}|^qd\mu_{\omega})^{\frac{2}{q}} <\infty$, 
by ergodic theorem, the above inequality becomes: for a.e. $\omega \in \Omega$,

\[||\sum_{i \le n}\Psi_{\sigma^i \omega} \circ F^i_{\omega}||_{L^q(\mu_{\omega})} \le C_q \cdot \sqrt{\sum_{i \le n}||\Psi^2_{\sigma^i \omega}||_{L^{\frac{q}{2}}(\mu_{\sigma^i\omega})}}= O_{\omega,q}(n^{\frac{1}{2}}).\]

Then for any $\delta \in (0, \frac{q}{2}-1) $,

\[\int |\frac{\sum_{i \le n}\Psi_{\sigma^i \omega} \circ F^i_{\omega}}{n^{\frac{1}{2}+\frac{1+\delta}{q}}}|^{q}d\mu_{\omega}=O_{\omega,q}(\frac{1}{n^{q \cdot (\frac{1+\delta}{q})}})=O_{\omega,q}(\frac{1}{n^{1+ \delta}}).\]

By Borel-Cantelli Lemma:

\[\sum_{i \le n}\Psi_{\sigma^i \omega} \circ F^i_{\omega}(x)=O_{x,\omega,q, \delta}(n^{\frac{1}{2}+\frac{1+\delta}{q}}) \text{ a.s.-}\mu_{\omega}.\]

\end{proof}

\begin{lemma}[Average vs Quenched]\label{atoq}\ \par
If $\psi \in L^{q}(\Delta, \mu) $ with $q \ge 2$ satisfies 
\[||P_{\omega}^n \psi_{\omega}||_{L^q(\Delta, \mu)}=O_{q, \psi}(\frac{1}{n^d})\text{ with } d>1,\]

then for any sufficiently small $\delta>0 $, a.e. $\omega \in \Omega$
\[||\sum_{i \le n}\psi_{\sigma^i \omega}\circ F^i_{\omega}||_{L^q(\mu_{\omega})}=O_{\omega, \psi, q}(n^{\frac{1}{2}}),\]
\[\sum_{i \le n}\psi_{\sigma^i \omega} \circ F^i_{\omega}(x)=O_{x,\omega,q, \delta, \psi}(n^{\frac{1}{2}+\frac{1+\delta}{q}}) \text{ a.s.-}\mu_{\omega}.\]
\end{lemma}

\begin{proof}
Let 
\[g_{\omega}:= \sum_{i \ge  0} P^i_{\sigma^{-i}\omega} ( \psi_{\sigma^{-i} \omega}),\]
\[\Psi_{\omega}:=\psi_{\sigma \omega} \circ F_{\omega} - g_{\sigma \omega}\circ F_{\omega}+g_{\omega}, \]
then 

\[\sum_{1 \le i \le n}\psi_{\sigma^i \omega} \circ F^i_{\omega}= \sum_{1 \le i \le n}\Psi_{\sigma^{i-1} \omega} \circ F^{i-1}_{\omega}+ g_{\sigma^n \omega} \circ F^{n}_{\omega}-g_{\omega}.\]

Since
\[||g||_{L^{q}(\Delta, \mu)}\le \sum_{i \ge 0} ||P^i_{\sigma^{-i}\omega}(\psi_{\sigma^{-i}\omega})||_{L^{q}(\Delta, \mu)}=O_{q, \psi} (\sum_{i \ge 1} \frac{1}{n^d}) < \infty, \]
then by Lemma \ref{momentcontrol}, 
\[ ||g_{\sigma^n \omega} \circ F^n_{\omega}||_{L^q(\mu_{\omega})}=O_{\omega,q}(n^{\frac{1}{q}}).\]

By (\ref{conditional2}), for a.e. $\omega \in \Omega$,
\[P_{\omega}\Psi_{\omega}=P_{\omega}(\psi_{\sigma \omega} \circ F_{\omega}) - P_{\omega}(g_{\sigma \omega}\circ F_{\omega})+P_{\omega}g_{\omega}=\psi_{\sigma \omega}-g_{\sigma \omega}+P_{\omega}g_{\omega}\]
\[=\psi_{\sigma \omega} - \sum_{i \ge  0} P^i_{\sigma^{-i}\sigma \omega} ( \psi_{\sigma^{-i} \sigma \omega})+\sum_{i \ge  0} P^{i+1}_{\sigma^{-i}\omega} ( \psi_{\sigma^{-i} \omega})=0,\]
then by (\ref{conditional}), 
\[\mathbb{E}_{\mu_{\omega}}(\Psi_{\sigma^{i} \omega} \circ F^{i}_{\omega}|(F_{\omega}^{i+1})^{-1}\mathcal{B}_{\sigma^{i+1} \omega})=[P_{\sigma^i \omega}(\Psi_{\sigma^i \omega})]\circ F_{\omega}^{i+1}=0, \]
that is, $(\Psi_{\sigma^{i} \omega} \circ F^{i}_{\omega})_{i \ge 0}$ is reverse martingale difference w.r.t.  $((F_{\omega}^{i})^{-1}\mathcal{B}_{\sigma^{i} \omega})_{i \ge 0}$.

Then by Lemma \ref{ratemartingale}, 

\[||\sum_{i \le n}\Psi_{\sigma^i \omega} \circ F^i_{\omega}(x)||_{L^q(\mu_{\omega})} = O_{\omega}(n^{\frac{1}{2}}).\]

Therefore,

\[||\sum_{1 \le i \le n}\psi_{\sigma^i \omega} \circ F^i_{\omega}||_{L^q(\mu_{\omega})}\le ||\sum_{1 \le i \le n}\Psi_{\sigma^{i-1} \omega} \circ F^{i-1}_{\omega}||_{L^q(\mu_{\omega})}+ ||g_{\sigma^n \omega} \circ F^{n}_{\omega}||_{L^q(\mu_{\omega})}+||g_{\omega}||_{L^q(\mu_{\omega})}\]
\[=O_{\omega}(n^{\frac{1}{2}})+O_{\omega}(n^{\frac{1}{q}})+O_{\omega}(1)=O_{\omega}(n^{\frac{1}{2}}).\]

Then for any $\delta \in (0, \frac{q}{2}-1) $,

\[\int |\frac{\sum_{i \le n}\psi_{\sigma^i \omega} \circ F^i_{\omega}}{n^{\frac{1}{2}+\frac{1+\delta}{q}}}|^{q}d\mu_{\omega}=O_{\omega,q}(\frac{1}{n^{q \cdot (\frac{1+\delta}{q})}})=O_{\omega,q}(\frac{1}{n^{1+ \delta}}).\]

By Borel-Cantelli Lemma:

\[\sum_{i \le n}\psi_{\sigma^i \omega} \circ F^i_{\omega}(x)=O_{x,\omega,q, \delta, \psi}(n^{\frac{1}{2}+\frac{1+\delta}{q}}) \text{ a.s.-}\mu_{\omega}.\]

\end{proof}

\begin{lemma}[Regularity]\label{regularity}\ \par
If $\phi \in \mathcal{F}^{\mathcal{K}}_{\beta,p}$ with Lipschitz constant $C_{\phi}$, then $(P^n_{\omega}\phi_{\omega})_{\omega \in \Omega} \in \mathcal{F}_{\beta,p}^{\mathcal{K}\circ \sigma^{-n}+C_{h,F}}$ with Lipschitz constant $C_{\phi}$ for any $n\in \mathbb{N}$, $C_{h,F}:=C_h+e^{C_h}C_F+e^{C_h+C_F}C_h$.

\end{lemma}

\begin{proof}
By Lemma \ref{produal}, $||P^n_{\omega}\phi_{\omega}||_{L^{\infty}(\mu_{\sigma^n \omega})} < \infty.$
By (\ref{dualexpress}), we have 

\[P^n_{\omega}\phi_{\omega}(x)=\frac{1}{h_{\sigma^n \omega}(x)}\sum_{F^n_{\omega}(z_x)=x}\frac{\phi_{\omega}(z_x)\cdot h_{\omega}(z_x)}{JF^n_{\omega}(z_x)}, \]
\[P^n_{\omega}\phi_{\omega}(y)=\frac{1}{h_{\sigma^n \omega}(y)}\sum_{F^n_{\omega}(z_y)=y}\frac{\phi_{\omega}(z_y)\cdot h_{\omega}(z_y)}{JF^n_{\omega}(z_y)}, \]
where $x,y \in \Delta_{\sigma^n \omega, l}$, $z_x, z_y \in I^{\omega}_{n-l} \in \bigvee^{n-l}_{j=0} (F^j_{\omega})^{-1}\mathcal{Z}_{\sigma^j \omega}$, $F_{\omega}^{n-l}:I^{\omega}_{n-l}\to \Delta_{\sigma^{n-l} \omega, 0} $ is bijective. Then

\[|P^n_{\omega}\phi_{\omega}(x)-P^n_{\omega}\phi_{\omega}(y)|=|\frac{1}{h_{\sigma^n \omega}(x)}\sum_{F^n_{\omega}(z_x)=x}\frac{\phi_{\omega}(z_x)\cdot h_{\omega}(z_x)}{JF^n_{\omega}(z_x)}\]
\[- \frac{1}{h_{\sigma^n \omega}(y)}\sum_{F^n_{\omega}(z_y)=y}\frac{\phi_{\omega}(z_y)\cdot h_{\omega}(z_y)}{JF^n_{\omega}(z_y)}|=|\frac{1}{h_{\sigma^n \omega}(x)}\sum_{F^n_{\omega}(z_x)=x}\frac{(\phi_{\omega}(z_x)-\phi_{\omega}(z_y))\cdot h_{\omega}(z_x)}{JF^n_{\omega}(z_x)}\]

\[+ \sum_{F^n_{\omega}(z_y)=y} \phi_{\omega}(z_y)\cdot (\frac{h_{\omega}(z_x)}{JF^n_{\omega}(z_x) \cdot h_{\sigma^n \omega}(x)}- \frac{h_{\omega}(z_y)}{JF^n_{\omega}(z_y) \cdot h_{\sigma^n \omega}(y)})|\]
\[\le C_{\phi} \cdot \mathcal{K}_{\omega} \cdot \beta^{s_{\sigma^n \omega}(x,y)}+C_{\phi} \cdot \sum_{F^n_{\omega}(z_y)=y}  \frac{h_{\omega}(z_y)}{JF^n_{\omega}(z_y) \cdot h_{\sigma^n \omega}(y)}\cdot |1-\frac{\frac{h_{\omega}(z_x)}{JF^n_{\omega}(z_x) \cdot h_{\sigma^n \omega}(x)}}{\frac{h_{\omega}(z_y)}{JF^n_{\omega}(z_y) \cdot h_{\sigma^n \omega}(y)}}|. \]

Using inequality $|1-z_1z_2z_3|\le |1-z_1|+|z_1||1-z_2|+|z_1||z_2||1-z_3|$, we have

\[\le C_{\phi} \cdot (\mathcal{K}\circ \sigma^{-n})_{\sigma^n \omega} \cdot \beta^{s_{\sigma^n \omega}(x,y)}+C_{\phi} \cdot C_{h,F}\cdot \beta^{s_{\sigma^n \omega}(x,y)} \]
\[\le C_{\phi} \cdot (C_{h,F}+\mathcal{K}\circ \sigma^{-n})_{\sigma^n \omega} \cdot \beta^{s_{\sigma^n \omega}(x,y)}. \]

\end{proof}

\section{Proof of Theorem \ref{QASIP}}\label{proof}

The equivariant probability measures $(\mu_{\omega})_{\omega \in \Omega}$ have been obtained in Lemma \ref{acma}. So it remains to prove Coboundary or QASIP and its convergence rate. Recall the conditions in Theorem \ref{QASIP}:

\[\phi \in \mathcal{F}_{\beta,p}^{\mathcal{K}}\text{ with } \int \phi_{\omega} d\mu_{\omega}=0,\]  

\[\rho_n= e^{-a\cdot n^{b}} \text{ or } \frac{1}{n^{D}} \text{ for some constants }\] 

\[ a>0, b \in (0,1], D>2+\frac{4\cdot p}{p-1}.\]

\subsection*{Martingale Decomposition}
\begin{lemma}[Decomposition]\label{martingaledecomp}\ \par
Let

\[g_{\omega}:= \sum_{i \ge  0} P^i_{\sigma^{-i}\omega} ( \phi_{\sigma^{-i} \omega}),\]
\[g(\omega,\cdot):=g_{\omega}(\cdot),\]
\[\psi_{\omega}:=\phi_{\sigma \omega} \circ F_{\omega} - g_{\sigma \omega}\circ F_{\omega}+g_{\omega}, \]
then for any sufficiently small $\delta>0$ s.t. $\psi, g \in L^{\frac{(D-2-\delta)\cdot(p-1)}{(\delta+1) \cdot p}}(\Delta, \mu) \subseteq L^4(\Delta, \mu)$. Besides, for a.e. $\omega \in \Omega$, we have decomposition:
\begin{equation}
    \sum_{1 \le i \le n}\phi_{\sigma^i \omega} \circ F^i_{\omega}= \sum_{1\le i \le n}\psi_{\sigma^{i-1} \omega} \circ F^{i-1}_{\omega}+ g_{\sigma^n \omega} \circ F^{n}_{\omega}-g_{\omega},
\end{equation}\label{decomposition}
where $(\psi_{\sigma^{i} \omega} \circ F^{i}_{\omega})_{i \ge 0}$ is a reverse martingale difference w.r.t.  $((F_{\omega}^{i})^{-1}\mathcal{B}_{\sigma^{i} \omega})_{i \ge 0}$. Moreover, 
\begin{enumerate}
    \item if $\rho_n=e^{-a\cdot n^b}$, $g \in L^{\infty}(\Delta, \mu)$,
    \item if $\rho_n=\frac{1}{n^D}$, $g \in L^{\frac{(D-2-\delta)\cdot(p-1)}{(1+\delta)\cdot p}}(\Delta, \mu)$  and 
    \[\int |g_{\sigma^n \omega} \circ F^n_{\omega}|^2 d\mu_{\omega}=O_{\omega,\delta}(n^{\frac{2(1+\delta)p}{(D-2-\delta)(p-1)}}) \text{ a.e. } \omega \in \Omega,\]
    \[g_{\sigma^n \omega} \circ F^n_{\omega}(x) =O_{\omega,x, \delta}(n^{\frac{(2+\delta)(1+\delta)p}{(D-2-\delta)(p-1)}}) \text{ a.s. } x \in \Delta_{\omega}.\] 
\end{enumerate}
\end{lemma}

\begin{proof}
Since $\phi \in L^{\infty}(\Delta, \mu)$, let $q:=\frac{(D-2-\delta)\cdot(p-1)}{(1+\delta)\cdot p}>4$ for sufficiently small $\delta$, by Lemma \ref{averagedecay} and (\ref{dualbound}),
\[||g||_{L^{q}(\Delta, \mu)}\le \sum_{i \ge 0} ||P^i_{\sigma^{-i}\omega}(\phi_{\sigma^{-i}\omega})||_{L^{q}(\Delta, \mu)}\]

\[\le C_{\phi}+ \sum_{i\ge 1} [\mathbb{E}\int |P^{i}_{\sigma^{-i}\omega}(\phi_{\sigma^{-i}\omega})|d\mu_{\omega}]^{\frac{1}{q}} \cdot C_{\phi}^{\frac{q-1}{q}}\]

\[\le C_{\phi}+\sum_{i \ge 1}  C_{\phi}^{\frac{q-1}{q}} \cdot  (C_{\phi} \cdot C_{h,F,\beta, \delta,p} \cdot ||\mathcal{K}||_{L^p})^{\frac{1}{q}}  \cdot \frac{1}{i^{1+\delta}} < \infty.\]

Therefore we have: if $\rho_n=e^{-a\cdot n^b}$, then $\rho_n \le \frac{1}{n^d}$ for any $d \gg 4$. So

\[||g||_{L^{\frac{(d-2-\delta)(p-1)}{(1+\delta)p}}(\Delta, \mu)}\le  C_{\phi}+\sum_{i \ge 1}  C_{\phi}^{\frac{d-1}{d}} \cdot  (C_{\phi} \cdot C_{h,F,\beta, \delta,p} \cdot ||\mathcal{K}||_{L^p})^{\frac{1}{d}}  \cdot \frac{1}{i^{1+\delta}}  \]

\[ \le C_{\phi}+\sum_{i \ge 1}  C_{\phi} \cdot  \max \{1, C_{\phi} \cdot C_{h,F,\beta, \delta,p} \cdot ||\mathcal{K}||_{L^p}\} \cdot \frac{1}{i^{1+\delta}} < \infty\]

Since the right hand side does not depend on $d$, so $g \in L^{\infty}(\Delta, \mu)$; if $\rho_n=\frac{1}{n^D}$, $g \in L^{q}(\Delta, \mu)$. Then by Lemma \ref{momentcontrol}, 
\[\int |g_{\sigma^n \omega} \circ F^n_{\omega}|^2 d\mu_{\omega}=O_{\omega,\delta}(n^{\frac{2(1+\delta)p}{(D-2-\delta)(p-1)}}) \text{ a.e. } \omega \in \Omega,\]
\[g_{\sigma^n \omega} \circ F^n_{\omega}(x) =O_{\omega,x, \delta}(n^{\frac{(2+\delta)(1+\delta)p}{(D-2-\delta)(p-1)}}) \text{ a.s. } x \in \Delta_{\omega}.\]

Since $\phi \in \mathcal{F}_{\beta,p}^{\mathcal{K}}$, so $\psi \in L^{q}(\Delta, \mu)$. By (\ref{conditional2}), for a.e. $\omega \in \Omega$,
\[P_{\omega}\psi_{\omega}=P_{\omega}(\phi_{\sigma \omega} \circ F_{\omega}) - P_{\omega}(g_{\sigma \omega}\circ F_{\omega})+P_{\omega}g_{\omega}=\phi_{\sigma \omega}-g_{\sigma \omega}+P_{\omega}g_{\omega}\]
\[=\phi_{\sigma \omega} - \sum_{i \ge  0} P^i_{\sigma^{-i}\sigma \omega} ( \phi_{\sigma^{-i} \sigma \omega})+\sum_{i \ge  0} P^{i+1}_{\sigma^{-i}\omega} ( \phi_{\sigma^{-i} \omega})=0,\]
then by (\ref{conditional}), 
\[\mathbb{E}_{\mu_{\omega}}(\psi_{\sigma^{i} \omega} \circ F^{i}_{\omega}|(F_{\omega}^{i+1})^{-1}\mathcal{B}_{\sigma^{i+1} \omega})=[P_{\sigma^i \omega}(\psi_{\sigma^i \omega})]\circ F_{\omega}^{i+1}=0, \]
that is, $(\psi_{\sigma^{i} \omega} \circ F^{i}_{\omega})_{i \ge 0}$ is reverse martingale difference w.r.t.  $((F_{\omega}^{i})^{-1}\mathcal{B}_{\sigma^{i} \omega})_{i \ge 0}$.
\end{proof}

\subsection*{Coboundary}
\begin{lemma}\label{errorestimate}
Define $\eta^2_{n}(\omega):=\int (\sum_{i \le n}\psi_{\sigma^{i} \omega} \circ F^{i}_{\omega})^2d\mu_{\omega}$. Then for sufficiently small $\delta>0$, a.e. $\omega\in \Omega$, 

\begin{equation}\label{vardiff}
    \sigma_n^2(\omega)-\eta_{n-1}^2(\omega)=O_{\omega, \delta}(n^{\frac{1}{2}+\frac{(1+\delta)\cdot p}{(D-2-\delta)\cdot(p-1)}}).
\end{equation}

Let $\sigma^2:=\mathbb{E}\int \psi_{\omega}^2 d\mu_{\omega}$, then a.e. $\omega \in \Omega$,
\[\lim_{n \to \infty}\frac{\sigma_n^2(\omega)}{n}=\lim_{n \to \infty}\frac{\eta_n^2(\omega)}{n}=\sigma^2.\]

If $\sigma^2=0$, then
\[\phi \circ F -g \circ F +g=0 \text{ a.s.-} \mu,\]
that is, 
\[\phi_{\sigma \omega} \circ F_{\omega}(x) - g_{\sigma \omega}\circ F_{\omega}(x)+g_{\omega}(x)=0 \text{ a.s.-}\mu.\]

If $\sigma^2>0$, there is a constant $C_{\omega} \in [1,\infty)$ s.t. $\eta^2_{n}(\omega)=C_{\omega}^{\pm} \cdot n$.

\end{lemma}
\begin{proof}
For sufficiently small $\delta>0$, let $q:=\frac{(D-2-\delta)\cdot(p-1)}{(1+\delta)\cdot p}>4$. Then by (\ref{decomposition}),

\[  \sigma_n^2(\omega)-\eta_{n-1}^2(\omega)=\int (\sum_{i \le n}\phi_{\sigma^i \omega} \circ F^i_{\omega})^2d\mu_{\omega}- \int (\sum_{i \le n}\psi_{\sigma^{i-1} \omega} \circ F^{i-1}_{\omega})^2d\mu_{\omega}\]
\[=\int (g_{\sigma^n \omega} \circ F^{n}_{\omega}-g_{\omega})\cdot ( g_{\sigma^n \omega} \circ F^{n}_{\omega}-g_{\omega}+2\sum_{i \le n}\psi_{\sigma^{i-1} \omega} \circ F^{i-1}_{\omega})d\mu_{\omega}\]

By Lemma \ref{martingaledecomp}, for a.e. $\omega \in \Omega$, 
 
\[=\int (g_{\sigma^n \omega} \circ F^{n}_{\omega}-g_{\omega})^2d\mu_{\omega}+2\int ( g_{\sigma^n \omega} \circ F^{n}_{\omega}-g_{\omega})\cdot(\sum_{i \le n}\psi_{\sigma^{i-1} \omega} \circ F^{i-1}_{\omega})d\mu_{\omega}\]
\[\le O_{\omega, \delta}(n^{\frac{2}{q}})+ O_{\omega, \delta}(n^{\frac{1}{q}}) \cdot ||\sum_{i \le n}\psi_{\sigma^{i-1} \omega} \circ F^{i-1}_{\omega}||_{L^2(\mu_{\omega})} \]

\[\le O_{\omega, \delta}(n^{\frac{2}{q}})+ O_{\omega, \delta}(n^{\frac{1}{q}+\frac{1}{2}})=O_{\omega, \delta}(n^{\frac{1}{q}+\frac{1}{2}}).\]

By (\ref{vardiff}) and ergodic theorem, for a.e. $\omega \in \Omega$,
\[\lim_{n \to \infty}\frac{\sigma^2_n(\omega)}{n}=\lim_{n \to \infty}\frac{\eta^2_{n}(\omega)}{n}=\mathbb{E} \int \psi^2_{\omega}d\mu_{\omega}=\sigma^2.\]

Then either $\mathbb{E} \int \psi^2_{\omega}d\mu_{\omega}=0$, that is, the Coboundary:
\[\phi_{\sigma \omega} \circ F_{\omega}(x) - g_{\sigma \omega}\circ F_{\omega}(x)+g_{\omega}(x)=0 \text{ a.s.-}\mu,\]
or $\mathbb{E} \int \psi^2_{\omega}d\mu_{\omega}>0$, then there is $C_{\omega}\ge 1$ s.t. 
\begin{equation}\label{lineargrow}
    \frac{1}{C_{\omega}}\le \frac{\eta^2_{n}(\omega)}{n}\le C_{\omega}.
\end{equation}

\end{proof}

\subsection*{Brownian motion Approximation}

From now on, we assume (\ref{lineargrow}), i.e. $\mathbb{E} \int \psi^2_{\omega}d\mu_{\omega}>0$:

\begin{lemma}[Brownian motion Approximation]\label{martingalasip}\ \par
 Let $\epsilon\in (0, \frac{1}{2}), \gamma:=\frac{1}{4\epsilon}$, define

\[R_n(\omega):=\sum_{i \ge n}\frac{\psi_{\sigma^{i} \omega} \circ F^{i}_{\omega}}{\eta^{2\gamma}_i(\omega)}, \delta_n^2(\omega):=\int R^2_n(\omega) d\mu_{\omega}, \]
where $\eta^2_{n}(\omega):=\int (\sum_{i \le n}\psi_{\sigma^{i} \omega} \circ F^{i}_{\omega})^2d\mu_{\omega}$. Then there is $C_{\omega, \gamma}\ge 1$ s.t.
\begin{equation}\label{varcompare}
   \delta^2_n(\omega)=C_{\omega, \gamma}^{\pm} \cdot \sigma_n^{2-4\gamma}(\omega) \to 0. 
\end{equation} 

Since $R_{n}(\omega)$ is $(F_{\omega}^{n})^{-1}\mathcal{B}_{\sigma^{n} \omega}$-measurable, so $(R_{n}(\omega))_{n \ge 0}$ is reverse martingale w.r.t.  $((F_{\omega}^{n})^{-1}\mathcal{B}_{\sigma^{n} \omega})_{n \ge 0}$. Therefore, by Theorem 2 in \cite{SH}, there is an extension $\bf{\Delta}_{\omega}$ of probability space $(\Delta_{\omega}, \mu_{\omega})$, Brownian motion $B^{\omega}$ defined on $(\bf{\Delta}_{\omega}, \bf{Q}_{\omega})$ and decreasing stopping time $\tau^{\omega}_i\searrow 0$, a constant $C_{\omega}\ge 1$ such that

\begin{equation}\label{sk1}
   R_n(\omega)=B^{\omega}_{\tau^{\omega}_n},
\end{equation}

\begin{equation}\label{sk2}
   \mathbb{E}_{\bf{Q}_{\omega}}[\tau^{\omega}_n-\tau^{\omega}_{n+1}|\mathcal{G}^{\omega}_{n+1}]=\mathbb{E}_{\mu_{\omega}}[\frac{\psi_n^2 \circ F_{\omega}^{n}}{\eta_n^{4\gamma}}|(F_{\omega}^{n+1})^{-1}\mathcal{B}_{\sigma^{n+1} \omega}],
\end{equation}

\begin{equation}\label{sk3}
    \mathbb{E}_{\mu_{\omega}}[\frac{\psi_n^{2q}\circ F^n_{\omega}}{\eta_n^{4q\gamma}}|(F_{\omega}^{n+1})^{-1}\mathcal{B}_{\sigma^{n+1} \omega}]= C_{\omega}^{\pm} \cdot \mathbb{E}_{\bf{Q}_{\omega}}[(\tau^{\omega}_n-\tau^{\omega}_{n+1})^q|\mathcal{G}^{\omega}_{n+1}],
\end{equation}

where $ q \ge 1$ and $\mathcal{G}^{\omega}_n=\sigma(\tau^{\omega}_i, (F_{\omega}^{i})^{-1}\mathcal{B}_{\sigma^{i} \omega}, i\geq n)$. Moreover, if 

\begin{equation}\label{qaispcondition}
\tau^{\omega}_n-\delta_n^2(\omega)=O(\delta_n^{2+2\epsilon}(\omega)) \text{ a.s.,}
\end{equation}

then 
\[  |\sum_{i \leq n} \psi_{\sigma^{i} \omega} \circ F^{i}_{\omega} -B^{\omega}_{\eta^2_n(\omega)}|=O(n^{\frac{1}{4}+\frac{3\epsilon-2\epsilon^3-\epsilon^2}{4}}) \text{ a.s.},\]

where the constants in $O(\cdot)$ depend on $\omega, \epsilon$ and $x \in \bf{\Delta}_{\omega}$.

\end{lemma}

\begin{proof}
To prove (\ref{varcompare}), first note that: by ergodic theorem, 
\[\frac{\eta^2_{n}(\omega)}{n}=\frac{\int (\sum_{i \le n}\psi_{\sigma^{i} \omega} \circ F^{i}_{\omega})^2d\mu_{\omega}}{n}=\frac{\sum_{i \le n}\int (\psi_{\sigma^{i} \omega})^2d\mu_{\sigma^i \omega}}{n} \to \mathbb{E}\int \psi_{\omega}^2d\mu_{\omega}>0 \text{ a.e. } \omega \in \Omega,\]
that is: there is $C_{\omega} \ge 1$ s.t. $\eta_{n}^2(\omega)=C_{\omega}^{\pm} \cdot n$, $\eta_{n+1}^2(\omega)=C_{\omega}^{\pm} \cdot \eta_{n}^2(\omega)$.

Second note that: since $(\psi_{\sigma^{i} \omega} \circ F^{i}_{\omega})_{i \ge 1}$ is reverse martingale difference, then

\[\delta^2_n(\omega)=\int R^2_{n}(\omega)d\mu_{\omega}=\sum_{i \ge n} \frac{\int \psi^2_{\sigma^{i} \omega} \circ F^{i}_{\omega}d\mu_{\omega}}{\eta_i^{4\gamma}(\omega)}=\sum_{i \ge n} \frac{\eta_i^{2}(\omega)-\eta_{i-1}^{2}(\omega)}{\eta_i^{4\gamma}(\omega)}\]

\[\le \int_{\eta^2_{n-1}(\omega)}^{\infty} \frac{1}{x^{2\gamma}} dx=\eta_{n-1}^{2-4\gamma}(\omega) \to 0,\]

\[\delta^2_n(\omega)=\sum_{i \ge n} \frac{\eta_i^{2}(\omega)-\eta_{i-1}^{2}(\omega)}{\eta_i^{4\gamma}(\omega)}\ge C_{\omega}^{-2\gamma} \cdot \sum_{i \ge n} \frac{\eta_i^{2}(\omega)-\eta_{i-1}^{2}(\omega)}{\eta_{i-1}^{4\gamma}(\omega)} \]

\[\ge  C_{\omega}^{-2\gamma}  \cdot \int_{\eta^2_{n-1}(\omega)}^{\infty} \frac{1}{x^{2\gamma}} dx =  C_{\omega}^{-2\gamma}  \cdot \eta_{n-1}^{2-4\gamma}(\omega) \to 0,\]

this proved (\ref{varcompare}). Now by (\ref{sk1}), we have

\[(B^{\omega}_{\tau^{\omega}_i}-B^{\omega}_{\tau^{\omega}_{i+1}})\cdot \eta_i^{2\gamma}(\omega)=\psi_{\sigma^i\omega}\circ F_{\omega}^i,\]
then 
\[  |\sum_{i \leq n} \psi_{\sigma^{i} \omega} \circ F^{i}_{\omega} -B^{\omega}_{\eta^2_n(\omega)}|=|\sum_{i \leq n} \psi_{\sigma^{i} \omega} \circ F^{i}_{\omega} -\sum_{i \le n }(B^{\omega}_{\delta^{2}_i}(\omega)-B^{\omega}_{\delta^{2}_{i+1}(\omega)})\cdot \eta_i^{2\gamma}(\omega)|\]

\[=|\sum_{i \leq n} [(B^{\omega}_{\tau^{\omega}_i}-B^{\omega}_{\tau^{\omega}_{i+1}})-(B^{\omega}_{\delta^{2}_i(\omega)}-B^{\omega}_{\delta^{2}_{i+1}(\omega)})]\cdot \eta_i^{2\gamma}(\omega)|\]
\[=|\sum_{ i \le n} (B^{\omega}_{\tau^{\omega}_i}-B^{\omega}_{\delta^{2}_{i}(\omega)}) \cdot (\eta_i^{2\gamma}(\omega)- \eta_{i-1}^{2\gamma}(\omega))- (B^{\omega}_{\tau^{\omega}_{n+1}}-B^{\omega}_{\delta^{2}_{n+1}(\omega)})\cdot \eta_n^{2\gamma}(\omega) |\]

By $\frac{1-\epsilon^2}{2}$-locally H\"older continuity of Brownian motion, the above identity becomes: there is $C_{\omega, \epsilon}>0$ s.t.

\[\le C_{\omega, \epsilon} \cdot [\sum_{ i \le n} |\tau^{\omega}_i-\delta^{2}_{i}(\omega)|^{\frac{1-\epsilon^2}{2}} \cdot (\eta_i^{2\gamma}(\omega)-\eta_{i-1}^{2\gamma}(\omega))+|\tau^{\omega}_{n+1}-\delta^{2}_{n+1}(\omega)|^{\frac{1-\epsilon^2}{2}} \cdot \eta_n^{2\gamma}(\omega)]\]

By (\ref{qaispcondition}), the above inequality becomes:
\[\le C_{\omega, \epsilon} \cdot O( [\sum_{ i \le n} |\delta_{i}^{2(1+\epsilon)}(\omega)|^{\frac{1-\epsilon^2}{2}} \cdot (\eta_i^{2\gamma}(\omega)-\eta_{i-1}^{2\gamma}(\omega))+|\delta_{n+1}^{2(1+\epsilon)}(\omega)|^{\frac{1-\epsilon^2}{2}} \cdot \eta_n^{2\gamma}(\omega)] )\]

Recall $\eta_{n}^2(\omega)=C_{\omega}^{\pm} \cdot n$, $\delta^2_{n}(\omega) \le \eta_{n-1}^{2-4\gamma}(\omega)$ and $\eta_{n+1}^2(\omega)=C_{\omega}^{\pm} \cdot \eta_{n}^2(\omega)$, the above inequality becomes:

\[\le C_{\omega, \epsilon} \cdot O( [\sum_{ i \le n} \eta_{i-1}^{(2-4\gamma) \cdot (1+\epsilon) \cdot \frac{1-\epsilon^2}{2}}(\omega)  \cdot (\eta_i^{2\gamma}(\omega)-\eta_{i-1}^{2\gamma}(\omega))+\eta_{n}^{(2-4\gamma) \cdot (1+\epsilon) \cdot \frac{1-\epsilon^2}{2}}(\omega)  \cdot \eta_n^{2\gamma}(\omega)] )\]

\[\le C_{\omega, \epsilon} \cdot O( [\sum_{ i \le n} \eta_{i}^{(2-4\gamma) \cdot (1+\epsilon) \cdot \frac{1-\epsilon^2}{2}}(\omega)  \cdot (\eta_i^{2\gamma}(\omega)-\eta_{i-1}^{2\gamma}(\omega))+\eta_{n}^{(2-4\gamma) \cdot (1+\epsilon) \cdot \frac{1-\epsilon^2}{2}}(\omega)  \cdot \eta_n^{2\gamma}(\omega)] )\]

\[\le C_{\omega, \epsilon} \cdot O([\int_0^{\eta_n^{2\gamma}(\omega)} x^{\frac{(2-4\gamma) \cdot (1+\epsilon) \cdot \frac{1-\epsilon^2}{2}}{2\gamma}}dx+\eta_n^{2\gamma+(2-4\gamma) \cdot (1+\epsilon) \cdot \frac{1-\epsilon^2}{2}}(\omega)] )\]

\[\le O(n^{\gamma-(2\gamma-1)\cdot (1+\epsilon)\cdot (\frac{1-\epsilon^2}{2})}) = O(n^{\frac{1}{4}+\frac{3\epsilon-2\epsilon^3-\epsilon^2}{4}}).\]

\end{proof}

\subsection*{Prove QASIP for Martingale}
Define $R_n(\omega):=\sum_{i \ge n}\frac{\psi_{\sigma^{i} \omega} \circ F^{i}_{\omega}}{\eta^{2\gamma}_i(\omega)}$. By Lemma \ref{martingalasip}, we will verify (\ref{qaispcondition}):

\[\tau^{\omega}_n-\delta_n^2(\omega)=O(\delta^{2(1+\epsilon)}_n(\omega))\]

\begin{lemma}[Stopping Time Decomposition]\label{reverseestimate}\ \par

We have the decomposition: 

 \[\tau^{\omega}_n-\delta_n^2(\omega):=R'_n(\omega)+R''_{n}(\omega)+S'_n(\omega) \]
 
 where \[S'_n(\omega)=\sum_{i \geq n} (\frac{\psi_{\sigma^i \omega}^2 \circ F_{\omega}^i}{\eta_i^{4\gamma}(\omega)}-\mathbb{E}_{\mu_{\omega}}\frac{\psi_{\sigma^i \omega}^2 \circ F_{\omega}^i}{\eta_i^{4\gamma}(\omega)}),\]
 
 \[R_n'(\omega):=\sum_{i \ge n} \tau^{\omega}_i-\tau^{\omega}_{i+1}-\mathbb{E}^{(F_{\omega}^{i+1})^{-1}\mathcal{B}_{\sigma^{i+1} \omega}}_{\mu_{\omega}}\frac{\psi_{\sigma^i \omega}^2 \circ F_{\omega}^i}{\eta_i^{4\gamma}(\omega)},\]
 
 \[R_n''(\omega):=\sum_{i\ge n} \mathbb{E}^{(F_{\omega}^{i+1})^{-1}\mathcal{B}_{\sigma^{i+1} \omega}}_{\mu_{\omega}}\frac{\psi_{\sigma^i \omega}^2 \circ F_{\omega}^i}{\eta_i^{4\gamma}(\omega)}-\frac{\psi_{\sigma^i \omega}^2 \circ F_{\omega}^i}{\eta_i^{4\gamma}(\omega)}\]
 
 $R'(\omega),R''(\omega)\text{ are reverse martingales}$ w.r.t. $((F_{\omega}^{i})^{-1}\mathcal{B}_{\sigma^{i} \omega})_{i \ge 1}$ and $( \mathcal{G}^{\omega}_i)_{i \ge 1}$ respectively  satisfying: for a.e. $\omega \in \Omega$, 
\[R_n^{'}(\omega)=O(\delta_n^{2+2\epsilon}(\omega)), R_n^{''}(\omega)=O(\delta_n^{2+2\epsilon}(\omega)),\]

where the constants in $O(\cdot)$ depend on $\omega, \epsilon$ and $x \in \bf{\Delta}_{\omega}$ and $\epsilon:=\frac{2p(1+\delta)^2}{(p-1)(D-2-\delta)}$ for sufficiently small $\delta$.

\end{lemma}

\begin{proof}
It is straightforward to verify the decomposition. To prove $R'_n(\omega), R''_n(\omega)$ are reverse martingales, we will verify the following two conditions:

Firstly,

\[R_n^{'}(\omega) \text{ and } R_n^{''}(\omega) \text{ are measurable w.r.t. } (F_{\omega}^{n})^{-1}\mathcal{B}_{\sigma^{n} \omega} \text{ and } \mathcal{G}^{\omega}_n \text{ respectively}.\]

Secondly, 
\[\mathbb{E}_{\mu_{\omega}}^{\mathcal{G}^{\omega}_{n+1}} R_n'(\omega)=\mathbb{E}^{\mathcal{G}^{\omega}_{n+1}}_{\mu_{\omega}}(\tau^{\omega}_n-\tau^{\omega}_{n+1})-\mathbb{E}^{\mathcal{G}^{\omega}_{n+1}}_{\mu_{\omega}}\mathbb{E}^{(F_{\omega}^{n+1})^{-1}\mathcal{B}_{\sigma^{n+1} \omega}}_{\mu_{\omega}}\frac{\psi_{\sigma^n \omega}^2 \circ F_{\omega}^n}{\eta_n^{4\gamma}(\omega)}\]

\[+\mathbb{E}^{\mathcal{G}^{\omega}_{n+1}}_{\mu_{\omega}}\sum_{i\ge n+1} \tau^{\omega}_i-\tau^{\omega}_{i+1}- \mathbb{E}^{(F_{\omega}^{i+1})^{-1}\mathcal{B}_{\sigma^{i+1} \omega}}_{\mu_{\omega}} \frac{\psi_{\sigma^i \omega}^2 \circ F_{\omega}^i}{\eta_i^{4\gamma}(\omega)}.\]

By (\ref{sk2}) and deceasing filtrations $((F_{\omega}^{n})^{-1}\mathcal{B}_{\sigma^{n} \omega})_{n \ge 0} \text{ and } (\mathcal{G}^{\omega}_n)_{n \ge 0}$ satisfying  $ (F_{\omega}^{n})^{-1}\mathcal{B}_{\sigma^{n} \omega} \subseteq \mathcal{G}^{\omega}_n $, the equality above becomes:

\[=\mathbb{E}^{(F_{\omega}^{n+1})^{-1}\mathcal{B}_{\sigma^{n+1} \omega}}_{\mu_{\omega}}\frac{\psi_{\sigma^n \omega}^2 \circ F_{\omega}^n}{\eta_n^{4\gamma}(\omega)}-\mathbb{E}^{(F_{\omega}^{n+1})^{-1}\mathcal{B}_{\sigma^{n+1} \omega}}_{\mu_{\omega}}\frac{\psi_{\sigma^n \omega}^2 \circ F_{\omega}^n}{\eta_n^{4\gamma}(\omega)}\]

\[+\mathbb{E}^{\mathcal{G}^{\omega}_{n+1}}_{\mu_{\omega}}\sum_{i\ge n+1} \tau^{\omega}_i-\tau^{\omega}_{i+1}- \mathbb{E}^{(F_{\omega}^{i+1})^{-1}\mathcal{B}_{\sigma^{i+1} \omega}}_{\mu_{\omega}} \frac{\psi_{\sigma^i \omega}^2 \circ F_{\omega}^i}{\eta_i^{4\gamma}(\omega)}=R''_{n+1}(\omega).\]

Since $((F_{\omega}^{n})^{-1}\mathcal{B}_{\sigma^{n} \omega})_{n \ge 0}$ is decreasing filtration, 

\[ \mathbb{E}^{(F_{\omega}^{n+1})^{-1}\mathcal{B}_{\sigma^{n+1} \omega}}_{\mu_{\omega}} R_n''(\omega)=\mathbb{E}^{(F_{\omega}^{n+1})^{-1}\mathcal{B}_{\sigma^{n+1} \omega}}_{\mu_{\omega}}\frac{\psi_{\sigma^n \omega}^2 \circ F_{\omega}^n}{\eta_n^{4\gamma}(\omega)}-\mathbb{E}^{(F_{\omega}^{n+1})^{-1}\mathcal{B}_{\sigma^{n+1} \omega}}_{\mu_{\omega}}\frac{\psi_{\sigma^n \omega}^2 \circ F_{\omega}^n}{\eta_n^{4\gamma}(\omega)}\]

\[+\mathbb{E}^{(F_{\omega}^{n+1})^{-1}\mathcal{B}_{\sigma^{n+1} \omega}}_{\mu_{\omega}}\sum_{i\ge n+1} \mathbb{E}^{(F_{\omega}^{i+1})^{-1}\mathcal{B}_{\sigma^{i+1} \omega}}_{\mu_{\omega}}\frac{\psi_{\sigma^i \omega}^2 \circ F_{\omega}^i}{\eta_i^{4\gamma}(\omega)}-\frac{\psi_{\sigma^i \omega}^2 \circ F_{\omega}^i}{\eta_i^{4\gamma}(\omega)}=R_{n+1}^{''}(\omega).\]

Therefore $R'(\omega),R''(\omega)\text{ are reverse martingales}$ w.r.t. $((F_{\omega}^{i})^{-1}\mathcal{B}_{\sigma^{i} \omega})_{i \ge 1}$ and $( \mathcal{G}^{\omega}_i)_{i \ge 1}$ respectively.

To prove the last statement, let $q:=\frac{(p-1)(D-2-\delta)}{2p(1+\delta)}>2$, $\epsilon:=\frac{2p(1+\delta)^2}{(p-1)(D-2-\delta)} \in (0, \frac{1}{2})$ for sufficiently small $\delta$.

By Burkholder-Davis-Gundy inequality and Minkowski inequality , there is constant $C_q$ s.t.

\[||R'_n(\omega)||_{L^q(\mu_{\omega})} \le C_q \cdot \sqrt{||\sum_{i \ge n}|\tau^{\omega}_i-\tau^{\omega}_{i+1}-\mathbb{E}^{(F_{\omega}^{i+1})^{-1}\mathcal{B}_{\sigma^{i+1} \omega}}_{\mu_{\omega}}\frac{\psi_{\sigma^i \omega}^2 \circ F_{\omega}^i}{\eta_i^{4\gamma}(\omega)}|^2||_{L^{\frac{q}{2}}(\mu_{\omega})}} \]

\[\le C_q \cdot \sqrt{ \sum_{i \ge n} ||\tau^{\omega}_i-\tau^{\omega}_{i+1}-\mathbb{E}^{(F_{\omega}^{i+1})^{-1}\mathcal{B}_{\sigma^{i+1} \omega}}_{\mu_{\omega}}\frac{\psi_{\sigma^i \omega}^2 \circ F_{\omega}^i}{\eta_i^{4\gamma}(\omega)}|^2||_{L^{\frac{q}{2}}(\mu_{\omega})}} \]

\[\le \sqrt{2} \cdot C_q \cdot \sqrt{ \sum_{i \ge n} |||\tau^{\omega}_i-\tau^{\omega}_{i+1}|^2+\mathbb{E}^{(F_{\omega}^{i+1})^{-1}\mathcal{B}_{\sigma^{i+1} \omega}}_{\mu_{\omega}}\frac{\psi_{\sigma^i \omega}^4 \circ F_{\omega}^i}{\eta_i^{8\gamma}(\omega)}||_{L^{\frac{q}{2}}(\mu_{\omega})}} \]

by (\ref{sk3}), the above inequality becomes

\[\le 2 \cdot C_q \cdot \sqrt{ \sum_{i \ge n} ||\frac{\psi_{\sigma^i \omega}^4 \circ F_{\omega}^i}{\eta_i^{8\gamma}(\omega)}||_{L^{\frac{q}{2}}(\mu_{\omega})}}= 2 \cdot C_q \cdot \sqrt{ \sum_{i \ge n} \frac{||\psi_{\sigma^i \omega}^4 \circ F_{\omega}^i||_{{L^{\frac{q}{2}}(\mu_{\omega})}}}{\eta_i^{8\gamma}(\omega)}} \]

Let $K_n(\omega) := \sum_{i \leq n}||\psi_{\sigma^i \omega}^4 \circ F_{\omega}^{i}||_{L^{\frac{q}{2}}(\mu_{\omega})}$, then by Lemma \ref{martingaledecomp} and $\frac{2}{q}<1$,

\[\mathbb{E}||\psi_{\omega}^4 ||_{L^{\frac{q}{2}}(\mu_{\omega})} \le (\mathbb{E} \int \psi_{\omega}^{2q} d\mu_{\omega})^{\frac{2}{q}} < \infty. \]

By ergodic theorem, for a.e. $\omega \in \Omega$, there is $C_{\omega} \ge 1$ such that
$K_{n}(\omega) = C_{\omega}^{\pm} \cdot n$ for all $n\in \mathbb{N}$.  Then the above inequality becomes:

\[  = 2 \cdot C_q \cdot \sqrt{ \sum_{i \ge n} \frac{K_i(\omega)-K_{i-1}(\omega)}{\eta_i^{8\gamma}(\omega)}} \le 2 \cdot C_q \cdot \sqrt{\frac{K_{n-1}(\omega)}{\eta_n^{8\gamma}(\omega)} + \sum_{i \geq n} K_i(\omega) \cdot \frac{(\eta_{i+1}^{8\gamma}(\omega)-\eta_{i}^{8\gamma}(\omega))}{\eta_i^{16\gamma}(\omega)}} \]

\[\le C_{\omega,q}\cdot \sqrt{\frac{n}{n^{4\gamma}} + \int_{\eta_n^{8}(\omega)}^\infty \frac{1}{x^{\frac{16\gamma-2}{8\gamma}}} dx} \le \sqrt{2} C_{\omega,q}\cdot \frac{1}{n^{\frac{4\gamma-1}{2}}}.\]

So by (\ref{varcompare}) and  $\gamma=\frac{1}{4\epsilon}$, there is $C_{\omega, q, \gamma}>0$ s.t.

\[||\frac{R'_n(\omega)}{\delta_{n}^{2+2\epsilon}(\omega)}||_{L^q(\mu_{\omega})}\le C_{\omega, q, \gamma} \cdot \frac{\eta_n^{(4\gamma-2)(1+\epsilon)}(\omega)}{n^{\frac{4\gamma-1}{2}}} \le C_{\omega, q, \gamma} \cdot \frac{n^{(2\gamma-1)(1+\epsilon)}}{n^{\frac{4\gamma-1}{2}}}= C_{\omega, q, \gamma} \cdot \frac{1}{n^{\epsilon}}.\]

So

\[||\frac{R'_n(\omega)}{\delta_{n}^{2+2\epsilon}(\omega)}||^q_{L^q(\mu_{\omega})} \le  C^q_{\omega, q, \gamma} \cdot \frac{1}{n^{q\epsilon}}= C^q_{\omega, q, \gamma} \frac{1}{n^{\frac{(p-1)(D-2-\delta)}{2p(1+\delta)}\cdot \frac{2p(1+\delta)^2}{(p-1)(D-2-\delta)}} }=C^q_{\omega, q, \gamma}\frac{1}{n^{1+\delta}}.\]

By Borel-Cantelli Lemma, 
\[R_n^{'}(\omega)=O(\delta_n^{2+2\epsilon}(\omega)) \text{ a.s..}\]

The estimate for $R^{''}_n(\omega)$ is similar.
\end{proof}

\begin{lemma}[Estimate $S'_n(\omega)$]\label{5.5}\ \par
Define $S_n(\omega):=\sum_{i\leq n}(\psi_{\sigma^i\omega}^2 \circ F_{\omega}^i-\int \psi_{\sigma^i\omega}^2 \circ F_{\omega}^id\mu_{\omega})$. If 
\[S_n(\omega)=O(\eta_n^{2-(4\gamma-2)\epsilon}(\omega))=O(n^{\frac{1}{2}+\epsilon})\text{ a.s.-}\mu_{\omega},\] 

then 

\[S'_n(\omega)=O(\delta_n^{2+2\epsilon}(\omega)) \text{ a.s.-}\mu_{\omega}.\]
\end{lemma}

\begin{proof}
\[S'_n(\omega)=\sum_{i \ge n} \frac{S_i(\omega)-S_{i-1}(\omega)}{\eta_i^{4\gamma}(\omega)}=-\frac{S_{n-1}(\omega)}{\eta_n^{4\gamma}(\omega)}+\sum_{i \ge n} S_i(\omega) \cdot (\frac{1}{\eta_i^{4\gamma}(\omega)}-\frac{1}{\eta_{i+1}^{4\gamma}(\omega)})\]

\[\le \frac{O(\eta_n^{2-(4\gamma-2)\epsilon}(\omega))}{\eta_n^{4\gamma}(\omega)}+ \sum_{i \ge n} O(\eta_i^{2-(4\gamma-2)\epsilon}(\omega)) \cdot \frac{\eta_{i+1}^{4\gamma}(\omega)-\eta_i^{4\gamma}(\omega)}{\eta_i^{8\gamma}(\omega)} \]
\[\le O(\frac{1}{\eta_n^{(4\gamma-2)(1+\epsilon)}(\omega)})+O(\int^{\infty}_{\eta_n^{4\gamma}(\omega)} \frac{1}{x^{\frac{8\gamma-2+(4\gamma-2)\epsilon}{4\gamma}}}dx) = O(\delta_n^{2(1+\epsilon)}(\omega)).\]
\end{proof}

To estimate $S'_n(\omega)$, we just need to estimate $S_n(\omega)$: since
\[\sum_{i \le n}\psi^2_{\sigma^{i-1}\omega} \circ F^{i-1}_{\omega}=\sum_{i \le n}(\phi_{\sigma^i \omega} \circ F_{\omega}^i-g_{\sigma^i \omega} \circ F^i_{\omega}+g_{\sigma^{i-1} \omega} \circ F^{i-1}_{\omega})^2\]
\[=\sum_{i \le n}\phi^2_{\sigma^i \omega}\circ F^i_{\omega}+g^2_{\sigma^i \omega} \circ F^i_{\omega}+g^2_{\sigma^{i-1} \omega} \circ F^{i-1}_{\omega}+2\phi_{\sigma^i \omega}\circ F^i\cdot g_{\sigma^{i-1} \omega} \circ F^{i-1}_{\omega} \]
\[-2\phi_{\sigma^i \omega}\circ F_{\omega}^i \cdot g_{\sigma^i \omega} \circ F^i_{\omega}-2g_{\sigma^{i-1} \omega} \circ F^{i-1}_{\omega}\cdot g_{\sigma^{i} \omega} \circ F^{i}_{\omega} \]
\[=\sum_{i \le n} \phi^2_{\sigma^i \omega}\circ F^i_{\omega}-g^2_{\sigma^i \omega} \circ F^i_{\omega}+g^2_{\sigma^{i-1} \omega} \circ F^{i-1}_{\omega}+2\phi_{\sigma^i \omega}\circ F_{\omega}^i\cdot g_{\sigma^{i-1} \omega} \circ F^{i-1}_{\omega} \]
\[+2g^2_{\sigma^i \omega} \circ F^i_{\omega}-2\phi_{\sigma^i \omega}\circ F_{\omega}^i \cdot g_{\sigma^i \omega} \circ F^i_{\omega}-2g_{\sigma^{i-1} \omega} \circ F^{i-1}_{\omega}\cdot g_{\sigma^{i} \omega} \circ F^{i}_{\omega} \]
\[=\sum_{i \le n} \phi^2_{\sigma^i \omega}\circ F^i_{\omega}-g^2_{\sigma^i \omega} \circ F^i_{\omega}+g^2_{\sigma^{i-1} \omega} \circ F^{i-1}_{\omega}+2\phi_{\sigma^i \omega}\circ F_{\omega}^i\cdot g_{\sigma^{i-1} \omega} \circ F^{i-1}_{\omega} \]
\[-2\psi_{\sigma^{i-1} \omega}\circ F_{\omega}^{i-1} \cdot g_{\sigma^i \omega} \circ F^i_{\omega} \]
\[=\sum_{i \le n} \phi^2_{\sigma^i \omega}\circ F^i_{\omega}-g^2_{\sigma^n \omega} \circ F^n_{\omega}+g^2_{ \omega} +2\sum_{i\le n}\phi_{\sigma^i \omega}\circ F_{\omega}^i\cdot g_{\sigma^{i-1} \omega} \circ F^{i-1}_{\omega} \]
\[-2\sum_{i\le n}\psi_{\sigma^{i-1} \omega}\circ F_{\omega}^{i-1} \cdot g_{\sigma^i \omega} \circ F^i_{\omega},\]
So we have
\[S_{n-1}(\omega)=\sum_{i \le n} (\phi^2_{\sigma^i \omega}\circ F^i_{\omega}-\mathbb{E}_{\mu_{\omega}} \phi^2_{\sigma^i \omega}\circ F^i_{\omega})\]
\[+2\sum_{i\le n}(\phi_{\sigma^i \omega}\circ F_{\omega}^i\cdot g_{\sigma^{i-1} \omega} \circ F^{i-1}_{\omega}-\int \phi_{\sigma^i \omega}\circ F_{\omega}^i\cdot g_{\sigma^{i-1} \omega} \circ F^{i-1}_{\omega}d\mu_{\omega})\]
\[-2\sum_{i\le n}(\psi_{\sigma^{i-1} \omega}\circ F_{\omega}^{i-1} \cdot g_{\sigma^i \omega} \circ F^i_{\omega}-\int \psi_{\sigma^{i-1} \omega}\circ F_{\omega}^{i-1} \cdot g_{\sigma^i \omega} \circ F^i_{\omega}d\mu_{\omega})\]
\[-g^2_{\sigma^n \omega} \circ F^n_{\omega}+g^2_{ \omega} +\mathbb{E}_{\mu_{\omega}}g^2_{\sigma^n \omega} \circ F^n_{\omega}-\mathbb{E}_{\mu_{\omega}}g^2_{ \omega}\]

\[\int \psi_{\sigma^{i-1} \omega}\circ F_{\omega}^{i-1} \cdot g_{\sigma^i \omega} \circ F^i_{\omega}d\mu_{\omega}=\int g_{\sigma^i \omega} \circ F^i_{\omega} \cdot \mathbb{E}_{\mu_{\omega}}[\psi_{\sigma^{i-1} \omega}\circ F_{\omega}^{i-1}|(F_{\omega}^i)^{-1}\mathcal{B}_{\sigma^i \omega}]d\mu_{\omega}=0,\]

the above equality becomes

\begin{equation}\label{6}
 =\sum_{i \le n} (\phi^2_{\sigma^i \omega}\circ F^i_{\omega}-\mathbb{E}_{\mu_{\omega}} \phi^2_{\sigma^i \omega}\circ F^i_{\omega})   
\end{equation}
\begin{equation}\label{7}
    +2\sum_{i\le n}(\phi_{\sigma^i \omega}\circ F_{\omega}^i\cdot g_{\sigma^{i-1} \omega} \circ F^{i-1}_{\omega}-\int \phi_{\sigma^i \omega}\circ F_{\omega}^i\cdot g_{\sigma^{i-1} \omega} \circ F^{i-1}_{\omega}d\mu_{\omega})
\end{equation}
\begin{equation}\label{8}
    -2\sum_{i\le n}(\psi_{\sigma^{i-1} \omega}\circ F_{\omega}^{i-1} \cdot g_{\sigma^i \omega} \circ F^i_{\omega})
\end{equation}
\begin{equation}\label{9}
    -g^2_{\sigma^n \omega} \circ F^n_{\omega}+g^2_{ \omega} +\mathbb{E}_{\mu_{\omega}}g^2_{\sigma^n \omega} \circ F^n_{\omega}-\mathbb{E}_{\mu_{\omega}}g^2_{ \omega}
\end{equation}

To estimate $S_n(\omega)$, we will estimate (\ref{6}),(\ref{7}),(\ref{8}),(\ref{9}).

\begin{lemma}[Estimate (\ref{9})]\label{5.11}\ \par
Let $\epsilon=\max \{\frac{2(2+\delta)(1+\delta)p}{(D-2)(p-1)}-\frac{1}{2}, 0\} \in (0, \frac{1}{2})$, for sufficiently small $\delta>0$, a.e. $\omega\in \Omega$, 

\[(\ref{9})=O_{\omega, x, \delta}(n^{\frac{1}{2}+\epsilon}) \text{  a.s.- } \mu_{\omega}.\]

\end{lemma}
\begin{proof}
For sufficiently small $\delta$, $\frac{2(2+\delta)(1+\delta)p}{(D-2-\delta)(p-1)}\le \epsilon+ \frac{1}{2}$. By Lemma \ref{martingaledecomp}, 
\[(\ref{9})=O_{\omega,x, \delta}(n^{\frac{2(2+\delta)(1+\delta)p}{(D-2-\delta)(p-1)}}) +O_{\omega,x}(1)+O_{\omega,x, \delta}(n^{\frac{2(1+\delta)p}{(D-2-\delta)(p-1)}})+O_{\omega}(1)=O_{\omega,x, \delta}(n^{\frac{1}{2}+\epsilon}).\]

\end{proof}

\begin{lemma}[Estimate (\ref{6})]\label{5.8}\ \par
Let $\epsilon=\frac{p(1+\delta)^2}{(p-1)(D-2-\delta)}$, for sufficiently small $\delta>0$, a.e. $\omega\in \Omega$, 
\[(\ref{6})=O_{\omega, x, \delta}(n^{\frac{1}{2}+\epsilon}).\]
\end{lemma}

\begin{proof}
Let $q=\frac{(D-2-\delta)(p-1)}{(1+\delta)p}$, since $\phi^2_{(\cdot)}-\mathbb{E}_{\mu_{(\cdot)}} \phi^2_{(\cdot)} \in L^{\infty}(\Delta, \mu) \bigcap \mathcal{F}_{\beta,p}^{\mathcal{K}} \subseteq L^{q}(\Delta, \mu)$ with Lipschitz constant $2C_{\phi}^2$, then by Lemma \ref{averagedecay} and (\ref{dualbound}), there is $C_{\phi,q,p,h,F}>0$,
\[(\mathbb{E} \int |P^n_{\omega}(\phi^2_{\omega}-\mathbb{E}_{\mu_{\omega}} \phi^2_{\omega})|^q d\mu_{\omega})^{\frac{1}{q}} \le (2C^2_{\phi})^{\frac{q-1}{q}} \cdot (\mathbb{E} \int |P^n_{\omega}(\phi^2_{\omega}-\mathbb{E}_{\mu_{\omega}} \phi^2_{\omega})| d\mu_{\omega})^{\frac{1}{q}}\]

\[\le C_{\phi,q,p,h,F} \cdot \frac{1}{n^{\frac{(D-2-\delta)(p-1)}{qp}}}\le C_{\phi,q,p,h,F} \cdot \frac{1}{n^{1+\delta}}< \infty.\]

So by Lemma \ref{atoq}, for a.e. $\omega \in \Omega$, \[(\ref{6})=O_{\omega, x, \delta}(n^{\frac{1}{2}+\frac{1+\delta}{q}})=O_{\omega, x, \delta}(n^{\frac{1}{2}+\epsilon}).\]

\end{proof}
\begin{lemma}[Estimate (\ref{8})]\label{5.10}\ \par
For sufficiently small $ \delta$, let $\epsilon=\frac{2p(1+\delta)^2}{(p-1)(D-2-\delta)}\in (0,\frac{1}{2})$, for a.e. $\omega\in \Omega$, 

\[\sum_{i\le n}(\psi_{\sigma^{i-1} \omega}\circ F_{\omega}^{i-1} \cdot g_{\sigma^i \omega} \circ F^i_{\omega})=O_{\omega, x, \delta}(n^{\frac{1}{2}+\epsilon}) \text{ a.s.-} \mu_{\omega}.\]
\end{lemma}
\begin{proof}
From Lemma \ref{martingaledecomp}, for a.e. $\omega\in \Omega$, $(\psi_{\sigma^{i} \omega}\circ F_{\omega}^{i})_{i \ge 0}$ is reverse martingale difference, so is $(\psi_{\sigma^{i-1} \omega}\circ F_{\omega}^{i-1} \cdot g_{\sigma^i \omega} \circ F^i_{\omega})_{i \ge 1}$. Let $q=\frac{(D-2-\delta)(p-1)}{2(1+\delta)p}$, then by Lemma \ref{martingaledecomp} again and H\"older inequality,
\[(\mathbb{E} \int |\psi_{\omega}\circ F_{\omega} \cdot g_{\sigma \omega} \circ F_{\omega}|^q d\mu_{\omega})^{\frac{1}{q}} \le (\mathbb{E} \int | g_{\sigma \omega} \circ F_{\omega}|^{2q}d\mu_{\omega})^{\frac{1}{2q}}\cdot (\mathbb{E}\int |\psi_{\omega}\circ F_{\omega}|^{2q}d\mu_{\omega})^{\frac{1}{2q}}< \infty.\]

Then by Lemma \ref{ratemartingale},
\[(\ref{8})=O_{x,\omega,q, \delta}(n^{\frac{1}{2}+\frac{1+\delta}{q}})=O_{x,\omega,q, \delta}(n^{\frac{1}{2}+\epsilon}) \text{ a.s.-}\mu_{\omega}.\]

\end{proof}
 
\begin{lemma}[Estimate (\ref{7})]\label{5.9}\ \par
For sufficiently small $ \delta$, let $\epsilon=\frac{2p(1+\delta)^2}{(p-1)(D-2-\delta)}\in (0,\frac{1}{2})$, for a.e. $\omega\in \Omega$, 
\[\sum_{i\le n}(\phi_{\sigma^i \omega}\circ F_{\omega}^i\cdot g_{\sigma^{i-1} \omega} \circ F^{i-1}_{\omega}-\int \phi_{\sigma^i \omega}\circ F_{\omega}^i\cdot g_{\sigma^{i-1} \omega} \circ F^{i-1}_{\omega}d\mu_{\omega})=O_{x,\omega,q, \delta}(n^{\frac{1}{2}+\epsilon}) \text{ a.s.-}\mu_{\omega}.\]
\end{lemma}

\begin{proof}
For sufficiently small $\delta$, let $q=\frac{(D-2-\delta)(p-1)}{2(1+\delta)p}>2$, denote $\Phi_{\omega}:=\phi_{\sigma \omega}\circ F_{\omega}\cdot g_{ \omega}-\int \phi_{\sigma \omega}\circ F_{\omega}\cdot g_{ \omega} d\mu_{\omega}$. Then
\[||\Phi||_{L^q(\Delta, \mu)}\le 2C_{\phi} \cdot ||g||_{L^q(\Delta,\mu)} <\infty,\]

Therefore, by (\ref{conditional2}) and Minkowski inequality,
\[(\mathbb{E}\int |P^k_{\omega}(\Phi_{\omega})|^q d\mu_{\sigma^k \omega})^{\frac{1}{q}}=(\mathbb{E}\int |P_{\omega}^k[\phi_{\sigma \omega}\circ F_{\omega}\cdot g_{ \omega}-\int \phi_{\sigma \omega}\circ F_{\omega}\cdot g_{ \omega}d\mu_{\omega}]|^q d\mu_{\sigma^k \omega}))^{\frac{1}{q}}\]

\[\le \sum_{i \ge 0}(\mathbb{E}\int |P_{\omega}^k[\phi_{\sigma \omega}\circ F_{\omega}\cdot P^{i}_{\sigma^{-i}\omega}(\phi_{ \sigma^{-i}\omega})-\int \phi_{\sigma \omega}\circ F_{\omega}\cdot P^{i}_{\sigma^{-i}\omega}(\phi_{ \sigma^{-i}\omega})d\mu_{\omega}]|^q d\mu_{\sigma^k \omega}))^{\frac{1}{q}}\]

\[\le \sum_{i \ge 0} (\mathbb{E} \int |P_{\sigma \omega}^{k-1}[\phi_{\sigma \omega}\cdot P^{i+1}_{\sigma^{-i}\omega}(\phi_{ \sigma^{-i}\omega})-\int \phi_{\sigma \omega}\cdot P^{i+1}_{\sigma^{-i}\omega}(\phi_{ \sigma^{-i}\omega})d\mu_{\sigma \omega}]|^q d\mu_{\sigma^k \omega}))^{\frac{1}{q}}\]

\[\le \sum_{i<k} (\mathbb{E} \int |P_{\sigma \omega}^{k-1}[\phi_{\sigma \omega}\cdot P^{i+1}_{\sigma^{-i}\omega}(\phi_{ \sigma^{-i}\omega})-\int \phi_{\sigma \omega}\cdot P^{i+1}_{\sigma^{-i}\omega}(\phi_{ \sigma^{-i}\omega})d\mu_{\sigma \omega}]|^q d\mu_{\sigma^k \omega}))^{\frac{1}{q}}\]

\[+ \sum_{i \ge k} (\mathbb{E} \int |P_{\sigma \omega}^{k-1}[\phi_{\sigma \omega}\cdot P^{i+1}_{\sigma^{-i}\omega}(\phi_{ \sigma^{-i}\omega})-\int \phi_{\sigma \omega}\cdot P^{i+1}_{\sigma^{-i}\omega}(\phi_{ \sigma^{-i}\omega})d\mu_{\sigma \omega}]|^q d\mu_{\sigma^k \omega}))^{\frac{1}{q}}.\]

By (\ref{dualbound}),
\[||\phi_{\sigma \omega}\cdot P^{i+1}_{\sigma^{-i}\omega}(\phi_{ \sigma^{-i}\omega})||_{L^{\infty}(\mu_{\sigma \omega})}\le C_{\phi}^2,\]
then the above inequality becomes: there is $C_{\phi,q}>0$ s.t.

\[\le C_{\phi,q} \cdot \sum_{i < k} (\mathbb{E} \int |P_{\sigma \omega}^{k-1}[\phi_{\sigma \omega}\cdot P^{i+1}_{\sigma^{-i}\omega}(\phi_{ \sigma^{-i}\omega})-\int \phi_{\sigma \omega}\cdot P^{i+1}_{\sigma^{-i}\omega}(\phi_{ \sigma^{-i}\omega})d\mu_{\sigma \omega}]| d\mu_{\sigma^k \omega}))^{\frac{1}{q}}\]

\[+ C_{\phi,q} \cdot \sum_{i \ge k} (\mathbb{E} \int |P_{\sigma \omega}^{k-1}[\phi_{\sigma \omega}\cdot P^{i+1}_{\sigma^{-i}\omega}(\phi_{ \sigma^{-i}\omega})-\int \phi_{\sigma \omega}\cdot P^{i+1}_{\sigma^{-i}\omega}(\phi_{ \sigma^{-i}\omega})d\mu_{\sigma \omega}]| d\mu_{\sigma^k \omega}))^{\frac{1}{q}}\]

\[\le C_{\phi,q} \cdot \sum_{i < k} (\mathbb{E} \int |P_{\sigma \omega}^{k-1}[\phi_{\sigma \omega}\cdot P^{i+1}_{\sigma^{-i}\omega}(\phi_{ \sigma^{-i}\omega})-\int \phi_{\sigma \omega}\cdot P^{i+1}_{\sigma^{-i}\omega}(\phi_{ \sigma^{-i}\omega})d\mu_{\sigma \omega}]| d\mu_{\sigma^k \omega}))^{\frac{1}{q}}\]

\[+ C_{\phi, q} \cdot \sum_{i \ge k} (\mathbb{E} \int | P^{i+1}_{\sigma^{-i}\omega}(\phi_{ \sigma^{-i}\omega})|d\mu_{\sigma \omega}+\mathbb{E} \int | P^{i+1}_{\sigma^{-i}\omega}(\phi_{ \sigma^{-i}\omega})|d\mu_{\sigma \omega}]| )^{\frac{1}{q}}\]

To proceed the estimate, we need to find the regularity of $\phi_{\sigma \omega}\cdot P^{i+1}_{\sigma^{-i}\omega}(\phi_{ \sigma^{-i}\omega})$:

for any $x,y\in \Delta_{\sigma \omega}$, by Lemma \ref{regularity} and (\ref{dualbound}),
\[|\phi_{\sigma \omega}(x)\cdot P^{i+1}_{\sigma^{-i}\omega}(\phi_{ \sigma^{-i}\omega})(x)-\phi_{\sigma \omega}(y)\cdot P^{i+1}_{\sigma^{-i}\omega}(\phi_{ \sigma^{-i}\omega})(y)|\]
\[\le |\phi_{\sigma \omega}(x)-\phi_{\sigma \omega}(y)|\cdot C_{\phi}+|P^{i+1}_{\sigma^{-i}\omega}(\phi_{ \sigma^{-i}\omega})(y)-P^{i+1}_{\sigma^{-i}\omega}(\phi_{ \sigma^{-i}\omega})(x)|\cdot C_{\phi}\]
\[\le \mathcal{K}_{\sigma \omega} \cdot \beta^{s_{\sigma \omega}(x,y)} \cdot C^2_{\phi}+ 4 \mathcal{K}_{\sigma^{-i} \omega} \cdot \beta^{s_{\sigma \omega}(x,y)} \cdot C^2_{\phi}.\]

Then $\phi_{(\cdot)}\cdot P^{i+1}_{\sigma^{-(i+1)}(\cdot)}(\phi_{ \sigma^{-(i+1)}(\cdot)}) \in \mathcal{F}_{\beta,p}^{(\mathcal{K}+4\mathcal{K}\circ \sigma^{-(i+1)})}$ with Lipschitz constant $C_{\phi}^2$.  So by Lemma \ref{averagedecay}, we can continuous our estimate: there is constant $C=C^2_{\phi} \cdot C_{h,F,\beta, \delta,p} \cdot ||\mathcal{K}+4\mathcal{K}\circ \sigma^{-(i+1)}||_{L^p}\le 5C^2_{\phi} \cdot C_{h,F,\beta, \delta,p} \cdot ||\mathcal{K}||_{L^p}$ s.t.

\[\le C \cdot C_{\phi, q}^{\frac{1}{q}}\cdot \sum_{i < k} \frac{1}{(k-1)^{\frac{(D-2-\delta)(p-1)}{qp}}}+C \cdot C_{\phi, q}^{\frac{1}{q}} \cdot \sum_{i \ge k}\frac{1}{i^{\frac{(D-2-\delta)(p-1)}{qp}}} \]

\[\le 2 C \cdot C_{\phi, q}^{\frac{1}{q}}\cdot \frac{1}{k^{\frac{(D-2-\delta)(p-1)}{qp}-1}}= 2 C \cdot C_{\phi, q}^{\frac{1}{q}}\cdot \frac{1}{k^{2\delta+1}}.\]

By Lemma \ref{atoq},  

\[\sum_{i\le n}(\phi_{\sigma^i \omega}\circ F_{\omega}^i\cdot g_{\sigma^{i-1} \omega} \circ F^{i-1}_{\omega}-\int \phi_{\sigma^i \omega}\circ F_{\omega}^i\cdot g_{\sigma^{i-1} \omega} \circ F^{i-1}_{\omega}d\mu_{\omega})\]

\[=O_{x,\omega,q, \delta}(n^{\frac{1}{2}+\frac{1+\delta}{q}})=O_{x,\omega,q, \delta}(n^{\frac{1}{2}+\epsilon}) \text{ a.s.-}\mu_{\omega}.\]

\end{proof}

So we have the following summary:

\begin{lemma}[QASIP for $(\psi_{\sigma^i \omega} \circ F^i_{\omega})_{i \ge 0}$]\label{asipmarting}\ \par

For sufficiently small $ \delta$, let $\epsilon=\frac{2p(1+\delta)^2}{(p-1)(D-2-\delta)}\in (0,\frac{1}{2})$, then for a.e. $\omega\in \Omega$, 
\[  |\sum_{i \leq n} \psi_{\sigma^{i} \omega} \circ F^{i}_{\omega} -B^{\omega}_{\eta^2_n(\omega)}|=O(n^{\frac{1}{4}+\frac{3\epsilon-2\epsilon^3-\epsilon^2}{4}}) \text{ a.s.}.\]

\end{lemma}
\begin{proof}
By Lemma \ref{reverseestimate}, \ref{5.5}, \ref{5.11}, \ref{5.8}, \ref{5.10}, \ref{5.9}, for sufficiently small $\delta$, we have: for a.e. $\omega\in \Omega$,
\[\tau^{\omega}_n-\delta_n^2(\omega)=O(\delta_n^{2+2\epsilon}(\omega)) \text{ a.s.,}\]

with $\epsilon=\max \{\frac{2p(1+\delta)^2}{(p-1)(D-2-\delta)}, \max \{\frac{2(2+\delta)(1+\delta)p}{(D-2)(p-1)}-\frac{1}{2}, 0\}, \frac{p(1+\delta)^2}{(p-1)(D-2-\delta)}\}=\frac{2p(1+\delta)^2}{(p-1)(D-2-\delta)}$.

By Lemma \ref{martingalasip},
\[  |\sum_{i \leq n} \psi_{\sigma^{i} \omega} \circ F^{i}_{\omega} -B^{\omega}_{\eta^2_n(\omega)}|=O(n^{\frac{1}{4}+\frac{3\epsilon-2\epsilon^3-\epsilon^2}{4}}) \text{ a.s.}.\]

\end{proof}

Now we can prove the final statement (\ref{matching}) of QASIP:

\begin{lemma}[QASIP for $(\phi_{\sigma^i \omega} \circ F^i_{\omega})_{i \ge 0}$]\ \par
 then for a.e. $\omega\in \Omega$, 
\[  |\sum_{i \leq n} \phi_{\sigma^{i} \omega} \circ F^{i}_{\omega} -B^{\omega}_{\sigma^2_n(\omega)}|=O(n^{\frac{1}{4}+\epsilon_0}) \text{ a.s.}.\]

\begin{enumerate}
\item if $\rho_n=e^{-a\cdot n^b}$, $\epsilon_0>0$ can be chosen to be any small number,
\item if $\rho_n=\frac{1}{n^D}$, $\epsilon_0$ can be chosen to be any number between $(\epsilon_D, \frac{1}{4})$, where
\[\epsilon_D=\max \{\frac{1}{4}+\frac{3\epsilon_1-2\epsilon_1^3-\epsilon_1^2}{4}, \epsilon_1, \frac{1+\epsilon_1}{4}\}-\frac{1}{4}\]
and 
\[\epsilon_1=\frac{2p}{(p-1)(D-2)}\in (0,\frac{1}{2}).\]
\end{enumerate}
\end{lemma}

\begin{proof}
Since $\epsilon_1=\frac{2p}{(p-1)(D-2)}\in (0, \frac{1}{2})$, so there is sufficiently small $\delta$ s.t.
\[ \epsilon=\frac{2p(1+\delta)^2}{(p-1)(D-2-\delta)}\in (\epsilon_1,\frac{1}{2}).\]
By Lemma \ref{martingaledecomp} and Lemma \ref{asipmarting},
\[\sum_{1 \le i \le n}\phi_{\sigma^i \omega} \circ F^i_{\omega}= \sum_{1\le i \le n}\psi_{\sigma^{i-1} \omega} \circ F^{i-1}_{\omega}+ g_{\sigma^n \omega} \circ F^{n}_{\omega}-g_{\omega}=\sum_{1\le i \le n}\psi_{\sigma^{i-1} \omega} \circ F^{i-1}_{\omega}\]
\[+O_{\omega,x, \delta}(n^{\frac{(2+\delta)(1+\delta)p}{(D-2-\delta)(p-1)}})+O_{\omega, x}(1)=B^{\omega}_{\eta^2_{n-1}(\omega)}+O(n^{\frac{1}{4}+\frac{3\epsilon-2\epsilon^3-\epsilon^2}{4}})+O_{\omega,x, \delta}(n^{\epsilon})+O_{\omega, x}(1)\]
\[=B^{\omega}_{\eta^2_{n-1}(\omega)}+O(n^{\max \{\frac{1}{4}+\frac{3\epsilon-2\epsilon^3-\epsilon^2}{4}, \epsilon\}}).\]

By (\ref{vardiff}) and basic property of Brownian motion,

\[B^{\omega}_{\eta^2_{n-1}(\omega)}=B^{\omega}_{\sigma^2_n(\omega)}+ O(n^{\frac{1}{4}+\frac{(1+\delta)^2\cdot p}{2(D-2-\delta)\cdot(p-1)}})=B^{\omega}_{\sigma^2_n(\omega)}+O(n^{\frac{1+\epsilon}{4}}).\]

Therefore, 

\[\sum_{1 \le i \le n}\phi_{\sigma^i \omega} \circ F^i_{\omega}=B^{\omega}_{\sigma^2_n(\omega)}+O(n^{\frac{1+\epsilon}{4}})+O(n^{\max \{\frac{1}{4}+\frac{3\epsilon-2\epsilon^3-\epsilon^2}{4}, \epsilon\}})\]

\[=B^{\omega}_{\sigma^2_n(\omega)}+O(n^{\max \{\frac{1}{4}+\frac{3\epsilon-2\epsilon^3-\epsilon^2}{4}, \epsilon, \frac{1+\epsilon}{4}\}})=B^{\omega}_{\sigma^2_n(\omega)}+O(n^{ \frac{1}{4}+\epsilon_0}),\]

where $\epsilon_0=\max \{\frac{1}{4}+\frac{3\epsilon-2\epsilon^3-\epsilon^2}{4}, \epsilon, \frac{1+\epsilon}{4}\}-\frac{1}{4} \in (\epsilon_D, \frac{1}{4})$.

If $\rho_n=e^{-a\cdot n^b}$, then $\rho_n \le \frac{1}{n^D}$ for sufficiently large $D$. Then $\epsilon_D$ is arbitrary closed to $0$, so is $\epsilon_0$.
\end{proof}

\section{Project From Tower}\label{fall}

In this section, we consider the RDS which can be extended to RYT:

\begin{definition}[Induced Random Markov Map]\label{fallcondition}\ \par
\begin{enumerate}[(1)]
    \item\label{A1} Assume Bernoulli scheme $(\Omega, \mathbb{P}, \sigma):=(I^{\mathbb{Z}}, {\nu}^{\mathbb{Z}}, \sigma)$  where $I$ is compact interval with normalized Lebegues probability measure $\nu$. $(M, \operatorname{Leb}, d)$ is compact Riemannian manifold with Riemannian volume $\operatorname{Leb}$ and Riemanian distant $d$. $(f_{\omega})_{\omega \in \Omega}$ is nonsingular random transformations w.r.t. $\operatorname{Leb}$ on $M$. Define:
    \[f_{\omega}^{n}:=f_{\sigma^{n-1} \omega} \circ f_{\sigma^{n-2} \omega} \circ \cdots \circ f_{\sigma \omega} \circ f_{\omega}.\] 
    \item\label{A2} Assume an open $\Lambda \subset M$, with normalized probability $m$ inherited from $\operatorname{Leb}$. 
    \item\label{A3} Assume for a.e. $\omega \in \Omega$, there are countable partition $ \mathcal{P}_{\omega} $ of a full measure subset $\mathcal{D}_{\omega}$ of $\Lambda$ and function $R_{\omega}: \Lambda \to \mathbb{N}$ such that $R_{\omega}$ is constant on each $U_{\omega} \in \mathcal{P}_{\omega}$, $\{x\in \Lambda: R_{\omega}(x) = n\}$ only depends on $\omega_{0}, \omega_{1}, \cdots, \omega_{n-1}$ and $f_{\omega}^{R_{\omega}}|_{U_{\omega}}$ is diffeomorphism from $U_{\omega}$ to $\Lambda$.
    \item\label{A4} Assume there are $N \in \mathbb{N}$, $\{\epsilon_i > 0,i = 1,\cdots, N\}$ and $\{t_i\in \mathbb{N},i= 1,\cdots,N\}$ with $\gcd(t_i) = 1 $ such that for a.e. $\omega \in \Omega$, all $1 \le i \le N$,
    \[\operatorname{Leb}(x \in \Lambda: R_{\omega}(x) = t_i) > \epsilon_i.\]
    \item\label{A5} Assume there are $\beta \in (0,1)$, constant $C>0$, random function $1 \le \mathcal{K}_{(\cdot)} \in L^{p}(\Omega)$ s.t. for a.e. $\omega \in \Omega$, any $U_{\omega} \in \mathcal{P}_{\omega}$, $x,y \in U_{\omega}$, and $0 \le k \le R_{\omega}|_{U_{\omega}}$:
    \begin{equation}\label{expansion}
      d(f_{\omega}^{R_{\omega}}(x), f_{\omega}^{R_{\omega}}(y)) \ge \beta^{-1} \cdot d(x,y),  
    \end{equation}
    \begin{equation}\label{distortion}
        |\log \frac{Jf_{\omega}^{R_{\omega}}(x)}{Jf_{\omega}^{R_{\omega}}(y)}| \le C \cdot d(f_{\omega}^{R_{\omega}}(x), f_{\omega}^{R_{\omega}}(y)),
    \end{equation}
    \begin{equation}\label{Lip}
       d(f_{\omega}^{k}(x), f_{\omega}^{k}(y)) \le C \cdot \mathcal{K}_{\sigma^k \omega} \cdot d(f_{\omega}^{R_{\omega}}(x), f_{\omega}^{R_{\omega}}(y)). 
    \end{equation}
   
    \item\label{A6}  Assume there is constant $C>0$ s.t. \[\int \operatorname{Leb}(x \in \Lambda: R_{\omega}(x)>n) d\mathbb{P} \le C \cdot \rho_n, \]
     where $\rho_n:= e^{-a\cdot n^{b}}$ or $\frac{1}{n^{D}}$ for some constant $a>0, b \in (0,1], D>2+\frac{4p}{p-\gamma}$, $\gamma \in (0,1]$ will be explained in Theorem \ref{qasiprds} below. 
\end{enumerate}
\end{definition}

\begin{theorem}[QASIP for RDS]\label{qasiprds}\ \par
Assume $(M, (f_{\omega})_{\omega \in \Omega}, \operatorname{Leb})$ satisfies the conditions in Definition \ref{fallcondition}.  Then for a.e. $\omega \in \Omega$, there are equivariant probability measures $(\upsilon_{\omega})_{\omega \in \Omega}$ on $M$, that is,

\[(f_{\omega})_{*} \upsilon_{\omega}=\upsilon_{\sigma \omega}.\]

For any H\"older function $\phi$ on $M$ with  H\"older exponent $\gamma \in (0,1]$. Define
\[{\varphi}_{\omega}:=\varphi -\int \varphi d\upsilon_{\omega},\]
\[ \sigma_n^2({\omega}): = \int(\sum_{k\le n} {\varphi}_{\sigma^k \omega} \circ f^k_{\omega} )^2 d\upsilon_{\omega}.\] 

Then $(M, (f_{\omega})_{\omega \in \Omega})$ satisfies the following:
\begin{enumerate}
    \item There is $\sigma^2 \ge 0$ s.t. $\lim_{n \to \infty}\frac{\sigma^2_n(\omega)}{n}=\sigma^2$ a.e. $\omega \in \Omega$.
    \item If $\sigma^2>0$, we have QASIP: there is $\epsilon_0 \in (0,\frac{1}{4})$ s.t. for a.e. $\omega \in \Omega$, there is Brownian motion $B^{\omega}$ defined on some extension of probability space $(M, \upsilon_{\omega})$, say $\bf{M}_{\omega}$, such that:
    \[\sum_{k\le n} {\varphi}_{\sigma^k \omega} \circ f^k_{\omega} -B^{\omega}_{\sigma^2_{n}(\omega)}=o(n^{\frac{1}{4}+\epsilon_0}) \text{ a.s.},\]
    
    where $\epsilon_0$ is the same as the one in (\ref{matching}).

\item Coboundary: define $\upsilon:= d\upsilon_{\omega}d\mathbb{P}(\omega)$, then there is measurable function $g$ defined on $(\bigcup_{\omega\in \Omega} (\{\omega\} \times M), \upsilon) $ s.t.
\[{\varphi}_{\sigma \omega} \circ f_{\omega}(x)=g_{\sigma \omega} \circ f_{\omega}(x)-g_{\omega}(x) \text{ a.s.-}\upsilon. \]

Moreover, if $\rho_n=e^{-a\cdot n^b}$, $g \in L^{\infty}(\upsilon)$; if $\rho_n=\frac{1}{n^D}$, $g \in L^{\frac{(D-2-\delta)\cdot(p-\gamma)}{(1+\delta)p}}(\upsilon)$ for sufficiently small $\delta$.
\end{enumerate}

\end{theorem}

To prove Theorem \ref{qasiprds}, we need one probability lemma:
\begin{lemma}[Transfer, see \cite{Ka} Theorem 6.10]\label{transfer}\ \par

Given probability spaces $(\Omega,\mathcal{F}, P), (\Omega',\mathcal{F}', P')$, $T$ is Borel space, $S$ is measurable space. Random elements $\eta': \Omega' \to T$, $\xi': \Omega' \to S $, $\xi: \Omega \to S$ with $\xi \stackrel{d}{=} \xi'$. Then there is measurable function $f: S \times [0,1] \to T$, if define random element $\eta: =f(\xi, U)$ with any uniform distribution on $[0,1]$: $U \sim U(0,1)$ independent of $\xi, f$, we have
\[(\eta, \xi) \stackrel{d}{=} (\eta', \xi').\]

One way to have $\xi$ independent of $U$ is defining $\eta$ on product probability space $(\Omega \times [0,1], P \times \operatorname{Leb}_{[0,1]})$, i.e. product extension of $(\Omega, P)$.

\end{lemma}

\begin{proof}[Proof of Theorem \ref{qasiprds}]\ \par
From Definition \ref{fallcondition}, there is RYT $(\Delta, F)$ such that $F_{\omega}^{R_{\omega}}=f_{\omega}^{R_{\omega}}$. Verify distortion (\ref{towerdistortion}) from (\ref{distortion}): if separation time $s_{\omega}(x,y)=n$, then for any $i < n$, $ F_{\omega}^{R^i_{\omega}(x)}(x), F_{\omega}^{R^i_{\omega}(y)}(y) $ lie in the same element of $\mathcal{P}_{\sigma^{R^i_{\omega}(x)} \omega}$ and $ F_{\omega}^{R^n_{\omega}(x)}(x), F_{\omega}^{R^n_{\omega}(y)}(y) $ lie in different elements of $\mathcal{P}_{\sigma^{R^n_{\omega}(x)} \omega}$. From (\ref{expansion}), we have
\[d(x,y) \le \beta \cdot d(f_{\omega}^{R_{\omega}}(x), f_{\omega}^{R_{\omega}}(y)) \le  \cdots \le \beta^n \cdot d(f_{\omega}^{R^n_{\omega}}(x), f_{\omega}^{R^n_{\omega}}(y)) \le \beta^n \cdot \sup_{x,y \in M}d(x,y).\]

From (\ref{distortion}), we have
\[ |\log \frac{JF^{R_{\omega}}_\omega(x)}{JF^{R_{\omega}}_\omega(y)}|\le C \cdot d(f_{\omega}^{R_{\omega}}(x), f_{\omega}^{R_{\omega}}(y)) \le C \cdot \sup_{x,y \in M}d(x,y) \cdot \beta^{n-1},\]
that is, there is $C>0$ s.t.
\[|\frac{JF^{R_{\omega}}_\omega(x)}{JF^{R_{\omega}}_\omega(y)}-1|\le C \cdot \beta^{s_{\sigma^{R_{\omega}(x)} \omega}(F^{R_{\omega}}_\omega(x), F^{R_{\omega}}_\omega(y))} \le C \cdot (\beta^{\gamma})^{s_{\sigma^{R_{\omega}(x)} \omega}(F^{R_{\omega}}_\omega(x), F^{R_{\omega}}_\omega(y))}  .\]

Therefore, by Lemma \ref{acma}, we construct a RYT $(\Delta, F, \mu)$ satisfying all assumptions in Definition \ref{ryt}. Define projection $\pi_{\omega}: \Delta_{\omega} \to M$ by $\pi_{\omega}(x,l):=f^l_{\sigma^{-l}\omega}(x)$. It is semi-conjugacy, $\upsilon_{\omega}:=(\pi_{\omega})_{*}\mu_{\omega}$ is equivariant probability measures, see section 3.1 in \cite{BBR}. $\phi_{\omega}:=\varphi_{\omega} \circ \pi_{\omega}$ is bounded above by $\max_{x \in M}|\varphi(x)|$ and fiberwise mean zero on $(\Delta_{\omega}, \mu_{\omega})$. We claim
\[\phi(\omega, \cdot):=\phi_{\omega}(\cdot) \in \mathcal{F}_{\beta^{\gamma}, \frac{p}{\gamma}}^{\mathcal{K}^{\gamma}}.\]

For any $(x,l), (y,l) \in \Delta_{\omega}$ with $s_{\omega}((x,l),(y,l))=n$, by (\ref{Lip}), 
\[|\phi_{\omega}(x,l)-\phi_{\omega}(y,l)|=|\varphi_{\omega}(f^l_{\sigma^{-l}\omega}x)-\varphi_{\omega}(f^l_{\sigma^{-l}\omega}y)|\le C_{\varphi} \cdot d(f^l_{\sigma^{-l}\omega}x, f^l_{\sigma^{-l}\omega}y)^{\gamma}\]
\[\le  C_{\varphi} \cdot C^{\gamma} \cdot \mathcal{K}^{\gamma}_{\omega} \cdot d(f_{\sigma^{-l}\omega}^{R_{\sigma^{-l}\omega}}(x), f_{\sigma^{-l}\omega}^{R_{\sigma^{-l}\omega}}(y))^{\gamma} \le  C_{\varphi} \cdot C^{\gamma} \cdot \mathcal{K}^{\gamma}_{\omega} \cdot (\beta^{\gamma})^{n-1} \cdot \sup_{x,y \in M}d(x,y)^{\gamma}\]
\[\le  C_{\varphi} \cdot C^{\gamma} \cdot \mathcal{K}^{\gamma}_{\omega} \cdot \sup_{x,y \in M}d(x,y)^{\gamma} \cdot (\beta^{\gamma})^{-1}  \cdot (\beta^{\gamma})^{s_{\omega}((x,l),(y,l))}.\]

So $\phi \in \mathcal{F}_{\beta^{\gamma}, \frac{p}{\gamma}}^{\mathcal{K}^{\gamma}} $ with Lipschitz constant $ C_{\varphi} \cdot C^{\gamma} \cdot \sup_{x,y \in M}d(x,y)^{\gamma} \cdot (\beta^{\gamma})^{-1}$. Apply Theorem \ref{QASIP}, we have: there is $\sigma^2 \ge 0$ s.t.

\[\lim_{n \to \infty} \frac{\int(\sum_{k\le n} {\varphi}_{\sigma^k \omega} \circ f^k_{\omega} )^2 d\upsilon_{\omega}}{n}=\lim_{n \to \infty}\frac{\int(\sum_{k\le n} \phi_{\sigma^k \omega} \circ F^k_{\omega} )^2 d\mu_{\omega}}{n}=\sigma^2.\]

This proves the first statement of Theorem \ref{qasiprds}. It remains to show how the QASIP or Coboundary of this RYT are projected to the RDS: let 

\[\tau(\omega, x):=(\sigma \omega, F_{\omega}(x)),\]
its transfer operator with respect to $\mu$ is $\tau^*$.

\[\upsilon:=d\upsilon_{\omega}d\mathbb{P}(\omega),\]
\[\chi(\omega,x):=(\sigma \omega, f_{\omega}(x)),\]
its transfer operator with respect to $\upsilon$ is $\chi^*$.
Define:

\[\varphi'(\omega, \cdot):=\varphi_{\omega}(\cdot).\]

For Coboundary, if we have the Coboundary on $(\Delta, \mu)$: 
\[\phi \circ \tau=g' \circ \tau-g' \text{ a.s.-}\mu,\]
where $g' \in L^{\frac{(D-2-\delta)\cdot(p-\gamma)}{(1+\delta)p}}(\Delta, \mu)$ (set $D=\infty$ if $\rho_n=e^{-a\cdot n^b}$), we want to show the Coboundary on $(\bigcup_{\omega\in \Omega} (\{\omega\} \times M), \upsilon) $: there is measurable function $g$ on $(\bigcup_{\omega\in \Omega} (\{\omega\} \times M), \upsilon)$ s.t. 

\[\varphi' \circ \chi = g \circ \chi-g \text{ a.s.-}\upsilon,\]
where $g \in L^{\frac{(D-2-\delta)\cdot(p-\gamma)}{(1+\delta)p}}(\bigcup_{\omega\in \Omega} (\{\omega\} \times M), \upsilon) $. 

To do this, we will verify the conditions of Theorem 1.1 in \cite{Li}:

Firstly, by ergodic theorem, a standard calculation gives
\[\lim_{n \to \infty }\frac{\int (\sum_{i\le n} \phi \circ \tau^i )^2 d\mu}{n}=\int \phi^2 d\mu+2\sum_{i \ge 1}\int \phi \cdot \phi \circ \tau^i d\mu,\]
while 
\[\lim_{n \to \infty }\frac{\int (\sum_{i\le n} \phi \circ \tau^i )^2 d\mu}{n}=\lim_{n \to \infty}\frac{\int(g'\circ \tau^{n+1}-g')^2d\mu}{n}\le \lim_{n \to \infty} \frac{2\int g'^2 d\mu}{n} \to 0,\]
then
\[\int \phi^2 d\mu+2\sum_{i \ge 1}\int \phi \cdot \phi \circ \tau^i d\mu=0.\]

Since $\phi(\omega, \cdot)=\varphi'(\omega, \pi_{\omega}(\cdot))$, so
\[\int \varphi'^2 d\upsilon+2\sum_{i \ge 1}\int \varphi' \cdot \varphi' \circ \chi^i d\upsilon=0.\]

Secondly, by Lemma \ref{averagedecay}, there is $C=C_{\phi, h,F,\beta^{\gamma}, \delta,\frac{p}{\gamma}} \cdot ||\mathcal{K}^{\gamma}||_{L^{\frac{p}{\gamma}}}$ s.t.
\[\sum_{i \ge 1}|\int \varphi' \cdot \varphi' \circ \chi^i d\upsilon|=\sum_{i \ge 1}|\int \phi \cdot \phi \circ \tau^i d\mu|=\sum_{i \ge 1} |\int \int \phi_{\omega}\cdot \phi_{\sigma^i \omega} \circ F^i_{\omega}d\mu_{\omega}d\mathbb{P}|\]
\[ \le C_{\phi} \cdot \sum_{i \ge 1} \int \int |P^i_{\omega}(\phi_{\omega})| d\mu_{\sigma^i\omega}d\mathbb{P}\le C \cdot \sum_{i \ge 1} \frac{1}{i^{(D-2-\delta)\cdot \frac{p-\gamma}{p}}} < \infty.\]

Thirdly,  by Lemma \ref{averagedecay}, there is $C=C_{\phi, h,F,\beta^{\gamma}, \delta,\frac{p}{\gamma}} \cdot ||\mathcal{K}^{\gamma}||_{L^{\frac{p}{\gamma}}}$ s.t. 
\begin{equation}\label{coboundarymonent}
  \sum_{n\ge 1}\int |(\chi^*)^n \varphi'| d\upsilon=\sum_{n\ge 1}\sup_{||\psi||_{L^{\infty}}\le 1}\int \psi \circ \chi^n \cdot  \varphi' d\upsilon  
\end{equation}

\[=\sum_{n\ge 1}\sup_{||\psi||_{L^{\infty}}\le 1}\int \int \psi_{\sigma^n \omega} \circ \pi_{\omega} \circ F_{\omega}^n \cdot  \phi_{\omega} d\mu_{\omega}d\mathbb{P}\]
\[\le \sum_{n\ge 1} \int \int | P_{\omega}^n (\phi_{\omega})| d\mu_{\sigma^n\omega}d\mathbb{P}\le  C \cdot \sum_{i \ge 1} \frac{1}{i^{(D-2-\delta)\cdot \frac{p-\gamma}{p}}} < \infty.\]

So by Theorem 1.1 in \cite{Li}, there is measurable function $g$ on $(\bigcup_{\omega\in \Omega} (\{\omega\} \times M), \upsilon)$ s.t. 

\[\varphi' \circ \chi = g \circ \chi-g \text{ a.s.-}\upsilon,\]
where $g:=\sum_{i \ge 0} (\chi^{*})^n \varphi' \in L^{\frac{(D-2-\delta)\cdot(p-\gamma)}{(1+\delta)p}}(\bigcup_{\omega\in \Omega} (\{\omega\} \times M), \upsilon)$, this is because, 
\[||g||_{L^{\frac{(D-2-\delta)\cdot(p-\gamma)}{(1+\delta)p}}} \le \sum_{n \ge 0} ||(\chi^{*})^n \varphi'||_{L^\frac{(D-2-\delta)\cdot(p-\gamma)}{(1+\delta)p}} \le C_{\varphi}+ \sum_{n \ge 1} (\int |(\chi^*)^n \varphi'| d\upsilon)^{\frac{(1+\delta)p}{(D-2-\delta)\cdot (p-\gamma)}}\]
using the same estimate of (\ref{coboundarymonent}), the above inquality becomes
\[\le C_{\varphi}+  C \cdot \sum_{n \ge 1} \frac{1}{n^{(D-2-\delta)\cdot \frac{p-\gamma}{p} \cdot \frac{(1+\delta)p}{(D-2-\delta)\cdot (p-\gamma)}}} \le C_{\varphi}+C \cdot \sum_{n \ge 1} \frac{1}{n^{1+\delta}} < \infty. \]

So we finish the proof of Coboundary.

For QASIP, by (\ref{matching}), we have Brownian motion $\bar{B}^{\omega}$ defined on an extended probability space $(\bf{\Delta}_{\omega}, \mathbb{Q}_{\omega})$ and 
\[\sum_{k\le n} \varphi_{\sigma^k \omega} \circ f^k_{\omega} \circ \pi_{\omega} -\bar{B}^{\omega}_{\sigma^2_{n}(\omega)}=\sum_{k\le n} \phi_{\sigma^k \omega} \circ F^k_{\omega} -\bar{B}^{\omega}_{\sigma^2_{n}(\omega)}=O(n^{\frac{1}{4}+\epsilon_0}) \text{ a.s.-}\mathbb{Q}_{\omega},\]
\[\sigma_n^2(\omega)=\int (\sum_{k\le n} \varphi_{\sigma^k \omega} \circ f^k_{\omega} \circ \pi_{\omega})^2 d\mu_{\omega}=\int (\sum_{k\le n} \varphi_{\sigma^k \omega} \circ f^k_{\omega})^2 d\upsilon_{\omega}.\]

By Lemma \ref{transfer}, there is a function $H: \mathbb{R}^{\mathbb{N}}\times [0,1] \to \mathbb{R}^{\mathbb{N}}$ and a Brownian motion $\hat{B}^{\omega}$ such that

\[\mathbb{Q}_{\omega}((( \varphi_{\sigma^k \omega} \circ f^k_{\omega} \circ \pi_{\omega})_{k \ge 1}, (\bar{B}^{\omega}_{\sigma^2_{k}(\omega)}-\bar{B}^{\omega}_{\sigma^2_{k-1}(\omega)})_{k \ge 1}) \in (\cdot, \cdot) )=\]
\[(\mu_{\omega} \times \operatorname{Leb}_{[0,1]})((( \varphi_{\sigma^k \omega} \circ f^k_{\omega} \circ \pi_{\omega})_{k \ge 1}, (\hat{B}^{\omega}_{\sigma^2_{k}(\omega)}-\hat{B}^{\omega}_{\sigma^2_{k-1}(\omega)})_{k \ge 1})\in (\cdot, \cdot)),\]

\[\sum_{k\le n} \varphi_{\sigma^k \omega} \circ f^k_{\omega} \circ \pi_{\omega} -\hat{B}^{\omega}_{\sigma^2_{n}(\omega)}=O(n^{\frac{1}{4}+\epsilon_0}) \text{ a.s.-}(\mu_{\omega}\times \operatorname{Leb}_{[0,1]}),\]
where $(\hat{B}^{\omega}_{\sigma^2_{k}(\omega)}-\hat{B}^{\omega}_{\sigma^2_{k-1}(\omega)})_{k \ge 1}=H((\varphi_{\sigma^k \omega} \circ f^k_{\omega} \circ \pi_{\omega})_{k \ge 1},U)$ forms a Brownian motion defined on $(\Delta_{\omega} \times [0,1], \mu_{\omega}\times \operatorname{Leb}_{[0,1]})$. Then $({B}^{\omega}_{\sigma^2_{k}(\omega)}-{B}^{\omega}_{\sigma^2_{k-1}(\omega)})_{k \ge 1}:=H(( \varphi_{\sigma^k \omega} \circ f^k_{\omega})_{k \ge 1},U)$ also forms a Brownian motion defined on $(M\times [0,1], \upsilon_{\omega}\times \operatorname{Leb}_{[0,1]})$. Hence
 
 \[\sum_{k\le n} \varphi_{\sigma^k \omega} \circ f^k_{\omega}-{B}^{\omega}_{\sigma^2_{n}(\omega)}=O(n^{\frac{1}{4}+\epsilon_0}) \text{ a.s.-}(\upsilon_{\omega}\times \operatorname{Leb}_{[0,1]}).\]
 
 Here ${\bf{M}}_{\omega}:=(M \times [0,1], \upsilon_{\omega}\times \operatorname{Leb}_{[0,1]})$.

\end{proof}

\section{Applications}\label{app}

We will apply Theorem \ref{qasiprds} to the following RDS via verifying the six conditions (\ref{A1})-(\ref{A6}) in Definition \ref{fallcondition}: i.i.d. translations of unimodal maps (satisfying Collet-Eckmann condition) in \cite{BBM}, non-uniformly expanding maps (with slow recurrence to singularities) in \cite{AV} or \cite{AH}, i.i.d. perturbations of admissible S-unimodal maps (satisfying Collet-Eckmann condition or summability condition of exponent $1$) in \cite{D} and LSV maps possessing an indifferent fixed point in \cite{BBR}. i.i.d. here means $f_{\sigma^i \omega}$ only depends on $\omega_i$, then for any $n\in \mathbb{N}$, $f_{\sigma^n \omega}$ is independent of $(f_{\sigma^i \omega})_{ i \le n-1}$. In Definition \ref{fallcondition}, conditions (\ref{A1}), (\ref{A2}), (\ref{A4}), (\ref{expansion}) and (\ref{distortion}) are naturally satisfied when random Young tower is constructed. Condition (\ref{A3}) is also satisfied since $\{R_{\omega}=n\}$ is constructed inductively, so only depends on  $(f_{\sigma^i \omega})_{0 \le i \le n-1}$, i.e. $\omega_0, \omega_1, \cdots, \omega_{n-1}$. So it remains to verify conditions (\ref{A6}) and (\ref{Lip}):
\begin{description}
\item[i.i.d. translations of unimodal maps, see \cite{BBM}]\ \par
Conditions (\ref{Lip}) holds due to Lemma 7.9, 7.10, 9.1. Condition (\ref{A6}) is due to Proposition 8.3 in \cite{BBM}.
\item[i.i.d. perturbations of S-unimodal maps, see \cite{D}]\ \par
Condition (\ref{A6}) is due to Theorem 8.1.2, 8.1.4 in \cite{D}. 
\begin{itemize}
    \item For S-unimodal maps satisfying Collet-Eckmann condition in \cite{D}, condition (\ref{Lip}) is due to Proposition 8.3.5 in \cite{D}.
    \item For S-unimodal maps on interval $I$ satisfying summability condition of exponent $1$ in \cite{D}, we will verify condition (\ref{Lip}) with $\mathcal{K}_{\cdot}\in L^{\infty}(\Omega)$: \[d(f_{\omega}^{k}(x), f_{\omega}^{k}(y)) \le C \cdot d(f_{\omega}^{R_{\omega}}(x), f_{\omega}^{R_{\omega}}(y)),\]
    where $k \le R_{\omega}=n, x,y\in U_{\omega}(z,n):=(f_{\omega}^n)^{-1}(\widetilde{B}(\delta)) \bigcap J^{\omega}_{z,n}$, $n$ is a $\theta-$good return time of $(\omega, z)$ into $\widetilde{B}(\delta)$,$f^n_{\omega}|_{U_{\omega}(z,n)}$ is diffeomorphism onto $\widetilde{B}(\delta)$. Equivalently, we will show
    \[|Df_{\sigma^k\omega}^{n-k}|_{f^n_{\omega}(U_{\omega}(z,n))}|\ge \frac{1}{C},\]
    where $n-k$ is a $\theta-$good return time of $(\sigma^k\omega, f^k_{\omega}(z))$ into $\widetilde{B}(\delta)$, $f^{n-k}_{\sigma^k\omega}$ is diffeomorphism from $f^n_{\omega}(U_{\omega}(z,n)) \subseteq U_{\sigma^k\omega}(f^k_{\omega}(z),n-k)$ onto $\widetilde{B}(\delta)$ (see Lemma 8.2.1, Proposition 8.2.3 and Proposition 8.2.4 in \cite{D}). By Lemma 8.2.1 in \cite{D} distortion: for any $z_1,z_2 \in U_{\sigma^k\omega}(f^k_{\omega}(z),n-k)$,
    \[e^{-\frac{1}{2}}\le \frac{|Df_{\sigma^k\omega}^{n-k}(z_1)|}{|Df_{\sigma^k\omega}^{n-k}(z_2)|} \le e^{\frac{1}{2}}.\]
    Then for any $z_1 \in f^n_{\omega}(U_{\omega}(z,n))$,
    \[|Df_{\sigma^k\omega}^{n-k}(z_1)|\ge e^{-\frac{1}{2}} \cdot \frac{|f_{\sigma^k\omega}^{n-k}(U_{\sigma^k\omega}(f^k_{\omega}(z),n-k))|}{|U_{\sigma^k\omega}(f^k_{\omega}(z),n-k)|}=\frac{|\widetilde{B}(\delta)|}{|I|}:=\frac{1}{C},\]
    so (\ref{Lip}) holds.
\end{itemize}
 
\item[i.i.d. translations of non-uniformly expanding maps, see \cite{AH}]\ \par
Condition (\ref{A6}) is due to Proposition 5.1, Section 5.2.2 and Theorem 2.9 in \cite{AH}. Condition (\ref{Lip}) is due to Proposition 4.9 in \cite{AH}.
\item[i.i.d. perturbations of LSV maps with neutral fixed point, see \cite{BBR}]\ \par
(\ref{Lip}) is because LSV maps have derivative no less than $1$, so $d(f_{\omega}^{k}(x), f_{\omega}^{k}(y)) \le d(f_{\omega}^{R_{\omega}}(x), f_{\omega}^{R_{\omega}}(y))$. Condition (\ref{A6}) is due to (5.5) and Proposition 5.3 in \cite{BBR} with $\frac{1}{\alpha_0}>6$, that is, QASIP holds for $\Omega=[\alpha_0,\alpha_1]^{\mathbb{Z}}$ where $0<\alpha_0<\frac{1}{6}, \alpha_1<1$.
\end{description}

\section*{Acknowledgments}
The author warmly thanks his advisor Prof. Andrew T\"or\"ok for posing him the question in this paper, and the support during the author studies at University of Houston. The author also thanks University of Houston for good place to study dynamical system.

%
\bibliographystyle{amsalpha}%
\bibliography{bibfile}

\end{document}